\documentclass[twoside,10pt]{article}

\usepackage{float}
\usepackage{jmlr2e}
\usepackage[latin1]{inputenc}
\usepackage[english]{babel}
\usepackage[T1]{fontenc}
\usepackage{amsmath,amsfonts}
\usepackage{graphicx}
\usepackage{color}
\usepackage{tikz}
\usepackage{hyperref}
\usepackage{algorithmic}
\usepackage[norelsize,ruled,vlined,commentsnumbered]{algorithm2e}

\usepackage{filecontents} 

\usepackage{color}
\definecolor{darkred}{RGB}{80,0,0}
\definecolor{darkgreen}{RGB}{0,100,0}
\definecolor{darkblue}{RGB}{0,0,100}
 \usepackage{hyperref}
\hypersetup{colorlinks=true, linkcolor=darkblue, citecolor=darkred, urlcolor=darkblue}

\ShortHeadings{Reconstructing Graphs from Eigenspaces}{De Castro, Espinasse and Rochet}
\firstpageno{1}

\newcommand{\supp}{\operatorname{Supp}}
\newcommand{\eq}{\begin{equation}}
\newcommand{\qe}{\end{equation}}

\newcommand{\vect}{\operatorname{vec}}
\newcommand{\id}{\operatorname{I}}

\newcommand{\Com}{\!\operatorname{Com}}

\newcommand{\R}{\mathbb{R}}
\newcommand{\W}{\mathsf{W}}
\newcommand{\K}{\mathsf{K}}
\newcommand{\A}{\mathsf{A}}

\tikzstyle{vertex}=[circle, draw, inner sep=0pt, minimum size=6pt]
\newcommand{\vertex}{\node[vertex]}

\newcommand{\Crit}{\operatorname{Crit}}

\newcommand{\independent}{\rotatebox[origin=c]{90}{$\models$}}

\begin{document}

\title{Reconstructing Undirected Graphs from Eigenspaces}

\title{Reconstructing Undirected Graphs from Eigenspaces}

\author{\name Yohann De Castro \email yohann.decastro@math.u-psud.fr \\
       \addr Laboratoire de Math\'ematiques d'Orsay, 
        Univ. Paris-Sud, CNRS, Universit\'e Paris-Saclay,\\
         F-91405 Orsay, France
       \AND
       \name Thibault Espinasse \email espinasse@math.univ-lyon1.fr\\
             \addr Institut Camille Jordan (CNRS UMR 5208), Universit\'e Claude Bernard Lyon 1,\\
         F-69622 Villeurbanne, France
       \AND
       \name Paul Rochet \email paul.rochet@univ-nantes.fr \\
             \addr Laboratoire de Math\'ematiques Jean Leray (CNRS UMR 6629), 
         Universit\'e de Nantes, \\
         F-44322 Nantes, France
     }

\editor{ArXiv:1603.08113v3}

\maketitle

\begin{abstract}
In this paper, we aim at recovering an undirected weighted graph of $N$ vertices from the knowledge of a perturbed version of the eigenspaces of its adjacency matrix $\W$. For instance, this situation arises for stationary signals on graphs or for Markov chains observed at random times. Our approach is based on minimizing a cost function given by the Frobenius norm of the commutator $\mathsf{A} \mathsf{B}-\mathsf{B} \mathsf{A}$ between symmetric matrices $\mathsf{A}$ and $\mathsf{B}$. 

In the Erd\H{o}s-R\'enyi model with no self-loops, we show that identifiability ({\it i.e.},~the ability to reconstruct $\W$ from the knowledge of its eigenspaces) follows a sharp phase transition on the expected number of edges with threshold function $N\log N/2$.

Given an estimation of the eigenspaces based on a $n$-sample, we provide support selection procedures from theoretical and practical point of views. In particular, when deleting an edge from the active support, our study unveils that our test statistic is the order of~$\mathcal O(1/n)$ when we overestimate the true support and lower bounded by a positive constant when the estimated support is smaller than the true support. This feature leads to a powerful practical support estimation procedure. Simulated and real life numerical experiments assert our new methodology.
\end{abstract}


\section{Presentation}
{Networks have become a natural and popular way to model interactions in applications such as information technology \citep{rossi2013activity}, social life \citep{jiang2013understanding,matias2015semiparametric}, genetics \citep{giraud2012graph}, ecology \citep{thomas2015chapter,miele2017revealing}.} In this paper, we investigate the reconstruction of an undirected weighted graph of size~$N$ from incomplete information on its set of edges (for instance, one knows that the target graph has no self-loops) and an estimation of the eigenspaces of its adjacency matrix~$\W$. This situation depicts any model where one knows in advance a linear operator~$\mathsf{K}$ that commutes with~$\W$. 

For instance, several authors \citep{espinasse2014parametric,girault2015stationary,perraudin2016stationary,marques2016stationary} have proposed a definition of stationarity for signal processing of graphs. In the Gaussian framework, they have shown that this definition implies that the covariance operator $\K$ is jointly diagonalizable with the Laplacian $\W$ \citep{perraudin2016stationary} or some weighted symmetric adjacency matrix supported on the graph \citep{espinasse2014parametric,marques2016stationary}. 

{Another framework adapted to our methodology concerns time-varying Markov processes, which are used to model numerous phenomena such as~chemical reactions \citep{anderson2011continuous} or waiting lines in queuing theory \citep{gaver1959imbedded}, see also \cite{MR668189,MR0468086,fytp14}.}  In some cases, one may observe at random times a Markov chain with transition matrix $\mathsf P$. The transition matrix $\mathsf Q$ of the resulting Markov chain can be shown to be a function of $\mathsf P$. Thus, the transitions on the original process can be recovered from an estimation of $\mathsf Q$ given that $\mathsf P$ and $\mathsf Q$ commute. Several models are presented in Section \ref{a:model} while the general model is given in Section~\ref{sec:Model}. 

Section \ref{sec:Identifiability} is concerned with identifiability issues, i.e.~the capacity to solve such problem. We exhibit sufficient and necessary conditions on the ability to reconstruct an undirected graph with no self-loops from the knowledge of the eigenspaces of $\W$. These conditions allow us to derive a sharp phase transition on identifiability in the Erd\H{o}s-R\'enyi model. 

Then, we introduce and theoretically assert new estimation schemes based on the Frobenius norm of the commutator $\mathsf{A} \mathsf{B}-\mathsf{B} \mathsf{A}$ between symmetric matrices $\mathsf{A}$ and $\mathsf{B}$, see Section~\ref{sec:commutator}. More precisely, we assume that we have access to an estimation $\widehat{\mathsf{K}}$ of ${\mathsf{K}}$ build from a $n$-sample and we consider the empirical contrast given by the commutator, namely $ \mathsf{A} \mapsto \Vert \widehat{\mathsf{K}} \mathsf{A} - \mathsf{A} \widehat{\mathsf{K}} \Vert $ where~$\Vert \cdot \Vert$ denotes the Frobenius norm. Using backward-type procedures based on this empirical contrast, Section \ref{sec:Estimation} derives estimators of the graph structure, {\it i.e.}, its set of edges~$S^\star$\textemdash referred to as the support. This studies reveals typical behaviors of the empirical contrast when the estimated support~$S$\textemdash referred to as the active set\textemdash contains or not the true support $S^\star$. Numerical experiments developed in {Section} \ref{sec:num} (simulated data) and Section \ref{sec:reallife} (real life data) assess the performances of our new estimation method. Discussion and related questions are presented in Section \ref{sec:discussion}.

To the best of our knowledge, the framework of this paper is new and the present results solve the identifiability issues and enforce an efficient backward-type estimation procedure. Related topics encompass spectral, least-squares and moment methods for graph reconstruction \citep{verzelen2015community,guedon2015community,klopp2015oracle,bubeck2016testing}, Graphical Models \citep{verzelen2008Gaussian,giraud2012graph}, or Vectorial AutoRegressive process \citep{hyvarinen2010estimation} to name but a few. In the specific cases of Ornstein-Uhlenbeck processes and non-linear diffusions, the interesting papers \cite{pereira2010learning} and  \cite{periera2014support} tackle a related subproblem that is to estimate~$\W$ along a trajectory, see Section~\ref{sec:SpatialAR} for further details. Note that the framework of the present paper addresses  processes observed at i.i.d.\! random times\textemdash with possibly unknown distribution\textemdash which are not covered by \cite{pereira2010learning} and  \cite{periera2014support}.

\section{Model and Identifiability}
\label{sec:Motivation}

\subsection{The Model}
\label{sec:Model}
Consider a symmetric matrix $\W\in\R^{N\times N}$ with some zero entries, where nonzero entries describe the intensity of a link of any form of local interaction. One may understand $\W$ as the adjacency matrix of an undirected weighted graph with $N$ vertices. We focus on the eigenspaces of~$\W$ examining models where we have no information on the spectrum of the graph. Depicting this situation, we assume that the information on the target $\W$ stems from an unknown transformation $\mathsf{K} =f(\W) \in\R^{N\times N}$ or, in more realistic scenarios, from a perturbed version $\widehat{\mathsf{K}}$ of $\mathsf{K}$. Here, $f$ is assumed to be an injective analytical function on the real line so that the transformation $\K = f(\W)$ may be understood as an operation on the spectrum of $\W$ only, stabilizing the eigenspaces. Therefore, $\W$ and $\K$ share the same eigenspaces and in particular, they commute, {\it i.e.}, $\W \K = \K \W$.\\

Our goal is to uncover $\W$ from the knowledge of an estimator $\widehat{\mathsf{K}}$ of $\mathsf{K}$, namely reconstruct $\W$ from a perturbed observation of its eigenspaces. The key point is then to use extra information given by the location of some zero entries of $\W$. Hence, we assume that one knows in advance a set $F\subset [1,N]^2$ of ``forbidden'' entries such that 
\eq
\tag{$\mathrm{\mathbf{H}_F}$}
\label{Hyp:Forbid}
\forall(i,j)\in F,\quad\W_{ij}=0\,
\qe
Equivalently, the set $F$ is disjoint to the set of edges of the target graph. Throughout this paper, a special case of interest is given by $F=F_{\mathrm{diag}}:=\{(i,i)\ :\ 1\leq i\leq N\}$ conveying that there are no self-loops in $\W$. 

\subsection{Identifiability}
\label{sec:Identifiability}
For $S \subseteq [1,N ]^2$, denote by $\mathcal E(S)$ the set of symmetric matrices $\mathsf{A}$ whose support is included in $S$, {which we write} $\supp(\mathsf{A}) \subseteq S$. Given the set $F$ of forbidden entries defined via \eqref{Hyp:Forbid}, the matrix of interest $\W$ is sought in the set $\mathcal E(\overline F)$ with $\overline F$ the complement of $F$. In some cases, typically for $F$ sufficiently large, most matrices $\mathsf{W} \in \mathcal E(\overline F)$ are uniquely determined by their eigenspaces. For those $\mathsf{W}\in\mathcal E(\overline F)$, there is no matrix $\mathsf{A}\in \mathcal E(\overline F)$ non collinear with~$\mathsf{W}$ that commutes with $\mathsf{W}$. This property is encapsulated by the notion of $F$-{\it identifiability} as follows.


\begin{definition}[$F$-identifiability]
\label{def:Fidentifiability}
We say that a {symmetric} matrix $\W$ is $F$-{\it identifiable} if, and only if, the only solutions $\mathsf{A}$ with $\supp(\mathsf{A}) \subseteq \overline F$ to $ \mathsf{A} \W = \W \mathsf{A}$ are of the form $\mathsf{A} =\lambda \mathsf{W}  $ for some~$\lambda \in \mathbb R$. Equivalently,
\eq
\label{eq:FindentifiabilityDefinition}
\big\{\mathsf A\in\mathbb R^{N\times N}\ :\ \mathsf A=\mathsf A^\top\!,\  \mathsf{A} \W = \W \mathsf{A}\ 
\mathrm{and}\ \supp(\mathsf{A}) \subseteq \overline F\big\}=\big\{\lambda\W\ :\ \lambda\in\mathbb R\big\} 
\qe
\end{definition}

\noindent
A matrix $\W$ is identifiable if the set of symmetric matrices with the same eigenvectors as~$\W$ and whose support is included in $\overline F$ is the line spanned by $\W$. 

\begin{remark} The dimension of the commutant, defined by
\[
\Com(\mathsf W) : = \Big\{ \mathsf A\in\mathbb R^{N\times N}\ :\ \mathsf A=\mathsf A^\top\!,\   \mathsf A \mathsf W = \mathsf W \mathsf A \Big\}\,,
\] 
is entirely determined by the multiplicity of the eigenvalues of $\mathsf W$. Indeed, letting $\lambda_1,\ldots,\lambda_s$ denote the different eigenvalues of $\mathsf W$ and $\ell_1,\ldots,\ell_s$ their multiplicities, one can show that 
\[
N\leq\dim \big(\Com(\mathsf W) \big) =  \sum_{j=1}^s \frac{\ell_j(\ell_j+1)}2\leq\frac{N(N+1)}2\,.
\]
Now, the $F$-identifiability of $\mathsf W$ can be stated equivalently as $\dim \big( \Com(\mathsf W) \cap \mathcal E(\overline F) \big) = 1$, observing that the left hand side of \eqref{eq:FindentifiabilityDefinition} is exactly $ \Com(\mathsf W) \cap \mathcal E(\overline F) $. Using a simple inclusion/exclusion formula, one can check that the condition
$$
|F|
\geq\dim \big(\Com(\mathsf W) \big) -1
$$ 
is necessary for the $F$-identifiability, where $|F|$ denotes the cardinality of $F$. In particular, a matrix $\mathsf W$ with repeated eigenvalues requires a large set $F$ of forbidden entries in order to be $F$-identifiable. 
\end{remark}

\noindent We have the following proposition.

\begin{proposition}[Lemma 2.1 in \cite{fytp14}]
\label{p:id}
Let $S\subseteq \overline F$, the set of $F$-identifiable matrices in $\mathcal E(S)$ is either empty or a dense open subset of $\mathcal E(S)$.
\end{proposition}

\noindent 
This proposition conveys that the $F$-identifiability of a matrix $\mathsf{W}$ is essentially a condition on its support $S$. The proof uses the fact that non $F$-identifiable matrices in $\mathcal E(S)$ can be expressed as the zeroes of a particular analytic function, we refer to \cite{fytp14} for further details. By abuse of notation, we say that a support $S\subseteq \overline F$ is $F$-{\it identifiable} if almost every matrix in~$\mathcal E(S)$ are $F$-identifiable. \\

Characterizing the $F$-identifiability appears to be a challenging issue since it can be viewed as understanding the eigen-structure of {graphs through their common support}. The particular case of the diagonal $F_{\mathrm{diag}}$ as the set of forbidden entries is given a particular attention in this paper. The $F_{\mathrm{diag}}$-identifiability, or diagonal identifiability, can be reasonably assumed in many practical situations since it entails that $\mathsf{W}$ lives on a simple graph, with no self-loops. In Theorem~\ref{t:iden} (see Appendix~\ref{sec:CNS}), we introduce necessary and sufficient conditions on the target support $\supp(\W)$ for diagonal identifiability. Defining the {\it kite} graph $\nabla_N$ of size $N \geq 3$ as the graph $(V,E)$ with vertices $V=[1,N]$ and edges $E=\{(k,k+1),\ 1\leq k\leq N-1\}\cup\{(N-2,N)\}$ (see Figure \ref{fig:kite}), one simple sufficient condition on diagonal identifiability reads as follows, a proof in given in Section~\ref{sec:kites_suffice}.
\begin{proposition}
\label{prop:KiteGeneral}
{If the graph $G=([1,N],S)$ contains the kite graph $\nabla_N$ as a subgraph, then $S$ is diagonally identifiable.}
\end{proposition}

\begin{figure}[ht]
\center
\begin{tikzpicture}[line width =0.3mm,scale=0.65]
\vertex[circle,  minimum size=22pt](1) at  (0,0){\small 1} ;
\vertex[circle,  minimum size=22pt](2) at  (2.6,0) {\small 2};
\vertex[circle,  minimum size=22pt](3) at  (5.2,0){\small 3} ;
\vertex[circle,  minimum size=0pt](3b) at (6.4,0){};
\vertex[circle,  minimum size=0pt](4b) at (7.4,0){};
\vertex[circle,  minimum size=22pt](4) at  (8.6,0){\small N-3} ;
\vertex[circle,  minimum size=22pt](5) at (11.2,0){\small N-2};
\vertex[circle,  minimum size=22pt](6) at (13.451,1.3){\small N-1};
\vertex[circle,  minimum size=22pt](7) at (13.451,-1.3){\small N};

\draw (1)--(2)--(3)--(3b);
\draw (4b)--(4)--(5)--(6)--(7)--(5) ;
\node[text=black] at (6.9,0) {.....};

\end{tikzpicture}
\caption{The kite graph $\nabla_N$ on $N$ vertices.}
\label{fig:kite}
\end{figure}
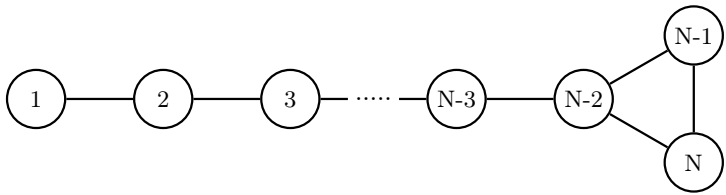

\noindent
Denote $G(N,p)$ the Erd\H{o}s-R\'enyi model on graphs of size $N$ where the edges are drawn independently with respect to the Bernoulli law of parameter $p$. Using Theorem \ref{t:iden}, one can prove that $\log N/N$ is a sharp threshold for diagonal identifiability in the Erd\H{o}s-R\'enyi model (see Section~\ref{sec:invertER}), it can be stated as follows.


\begin{theorem}
\label{t:invertER}
Diagonal identifiability in the Erd\H{o}s-R\'enyi model occurs with a sharp phase transition with threshold function $\log N/N$: for any $\varepsilon>0$, it holds
\begin{itemize}
\item If $p_N\geq(1+\varepsilon){\log N}/N$ and $G_N\sim G(N,p_N)$ then the probability that $\supp(G_N)$ is diagonally identifiable tends to $1$ as $N$ goes to infinity.
\item If $p_N\leq(1-\varepsilon){\log N}/N$ and $G_N\sim G(N,p_N)$ then the probability that $\supp(G_N)$ is diagonally identifiable tends to $0$ as $N$ goes to infinity.
\end{itemize}
\end{theorem}

\noindent
In practice, one may expect that any target graph of size $N$ with no self-loops and degree bounded from below by $\log N$ is diagonally identifiable. In this case, it might be recovered from its eigenspaces. Conversely, small degree graphs (\textit{i.e.},~graphs with some vertices of degree much smaller than $\log N$) may not be identifiable. In this case, there is no hope to reconstruct it from its eigenspaces since there exist another small degree undirected weighted graph with the same eigenspaces.

\section{Some Concrete Models}
\label{a:model}

\subsection{Markov chains} 
We begin with an example treated in the companion papers \cite{fytp14,Barsotti2016}. Consider a Markov chain $(X_n)_{n \in \mathbb N}$ with finite state space $[1,N]$ and transition matrix $\mathsf{P}\in\mathbb R^{N\times N}$. Let~$(T_k)_{k \geq 1}$ be a sequence of random times such that the time gaps $\tau_k: =T_{k+1}-T_k$ are i.i.d random variables independent of $(X_n)_{n \in \mathbb N}$. One can show that the sequence $Y_k = X_{T_k}$ is also a Markov chain with transition matrix $ \mathsf Q = \mathbb E[\mathsf P^{\tau_1}] =: f(\mathsf{P})$ where $f$ is the generating function of~$\tau_k$. Indeed, this follows from noticing that
\begin{align} \mathbb P[Y_{k+1} = j \vert Y_k = i] & = \mathbb P[X_{T_{k+1}} = j \vert X_{T_k} = i] \nonumber \\
& = \sum_{t\geq 0} \mathbb P[X_{T_{k}+t} = j, \tau_{k}=t\vert X_{T_k} = i] \nonumber \\
& = \sum_{t \geq 0} \mathbb P[X_{t} = j \vert X_0 = i] \mathbb P[\tau_k=t] \nonumber
\\
& = \sum_{t \geq 0}  \mathbb P[\tau_1=t](\mathsf{P}^t)_{i,j}. \nonumber
\end{align}
Under regularity conditions, $\mathsf{Q}=f(\mathsf{P})$ can be estimated and one may recover~$\mathsf{P}$ from $\mathsf{Q}$ without any information on the distribution of the time gaps $\tau_k$. 

\subsection{Vectorial AutoRegressive process}
 Consider a stationary Vectorial AutoRegressive process of order one $(X_n)_{n \in \mathbb Z}$ verifying
$$X_{n+1} = \W X_n + \varepsilon_n\,, $$
with $\varepsilon_i$ i.i.d.\! centered random variables. Define as above $Y_k = X_{T_k}$ where $T_k$ are random times such that the time gaps $\tau_k= T_{k+1}-T_{k}$ are i.i.d. with generating function $f$. Then, it holds 
\[
\mathbb E[Y_{k+1}|Y_{k}] = \mathbb E [ \mathbb E[Y_{k+1}|Y_{k}, \tau_k]  | Y_k]  =  \sum_{j=0}^\infty \mathsf W^j Y_{k}\mathbb P(\tau_k =j)  = f(\mathsf W) Y_k ,
\] 
which allows us to estimate $\mathsf{K} = f(\W )$ and ultimately recover $\W $ from this estimate. 

\subsection{Ornstein-Uhlenbeck process}
The same property holds for the continuous time version of this process, namely a vectorial Ornstein-Uhlenbeck process observed at random times verifying
$$\mathrm{d}X_t = \W X_t\mathrm{d}t+\mathrm{d}B_t. $$
In this case, one can check that the random process $Y_k := X_{T_k}$ where the $T_k$'s are random times with i.i.d.~gaps $\tau_k=T_{k+1}-T_{k}$ satisfies
$$\mathbb E[Y_{k+1}|Y_{k}] =f(\W )Y_{k}\,,$$
for $f$ the Laplace transform of $\tau_1$\textemdash that is $f(\W )=\mathbb E[\exp(-\tau_1\W)]$. This follows from observing that 
\[
\forall t,u\in\mathbb R,\quad\mathbb E[X_{t+u}|X_u] =\exp(-t\W)X_u\,
\]
so that 
\[\mathbb E[Y_{k+1}|Y_{k}] =\mathbb E [ \mathbb E[X_{T_{k}+\tau_k}|X_{T_{k}}, \tau_k]  | X_{T_{k}}] =\mathbb E [ \exp(-\tau_k\W)X_{T_{k}} | X_{T_{k}}] =\mathbb E[\exp(-\tau_k\W)]Y_{k}\,,
\] 
by independence of $\tau_k$ and $Y_{k-1}$. 

\subsection{Gaussian Graphical models}
Our model is related to Gaussian Graphical models for which an overview can be found in the thesis \cite{verzelen2008Gaussian}. The reader may also consult the pioneering paper \cite{friedman2008sparse}. One may consider $\W $ as the precision matrix, \textit{i.e.},~the inverse of the covariance matrix, having some non zero entries described by a graph of dependencies. Using $f(x)= {x}^{-1}$, this falls into our setting, trying to recover $\W $ from the estimation of the covariance matrix. Of course, in this case, it is better to use the knowledge of~$f$, which certainly improves estimation. However, our procedure allows us to estimate the function~$f$ and heuristically validate the hypothesis $f(x)= {x}^{-1}$.

\subsection{Seasonal $\mathrm{VAR}$ structure}
We can also consider a toy example looking at a seasonal $\mathrm{VAR}$ structure without any randomness on the times of observations. Let $T$ be a positive integer, and $(u_k)_{k \in \mathbb{Z}},(v_k)_{k \in \mathbb{Z}}$ be some periodic sequences of period $T$. Consider the following model
\[
\forall k \in \mathbb{Z}\,,\quad Y_{k+1} = u_k Y_k + v_k \W  Y_k+\varepsilon_k\,,
\]
where $\varepsilon_k$ are independent and centered random variables. We may observe the model only at time gap intervals $T$ with some error, {\it i.e.}, $X_t = Y_{tT+k_0}+ \eta_t$ with $\eta_t$ centered and independent random variables. This falls into the general frame
\[
\mathbb E[X_t|X_{t-1}] = f(\W ) X_{t-1} \quad\mathrm{where}\quad f(x) := \prod_{k =1}^T(u_k  -v_k x)\,.
\]
In this case, $\mathsf{K} = f(\W )$ can be estimated from the observations.

\subsection{Spatial autoregressive Gaussian fields}
\label{sec:SpatialAR}
Note that Gaussian autoregressive processes on $\mathbb{Z}$ verify that the precision operator may be written as a polynomial of the adjacency operator of $\mathbb{Z}$. One natural way to extend this property (see for instance \cite{espinasse2014parametric}) is to define centered Gaussian autoregressive fields on a graph through the same relation between the covariance operator $\K$ and the adjacency operator~$\W$ (or the discrete Laplacian, depending on the framework) : $\K^{-1} = P(\W)$, with $P$ a polynomial of degree~$d$. In this framework, Graphical models methods will infer the graph of path of length~$d$, whereas our methods aims to recover $\W$. Note that this framework extends to ARMA spatial fields where $\K$ writes as a rational fraction of $\W$, and the property of commutativity between $\W$ and $\K$ still holds. 

In the previous cases, we assumed that we can not estimate directly $\W$. For spatio-temporal processes, this means that we do not have access to a full trajectory. It may be the case when the sample is drawn at random times, or when we can only sample independently under stationary measure of such process, for instance when observation times are a lot larger than the typical evolution time's scale of the process. If the whole trajectory is available, it would be better to use this extra information, see for instance \cite{pereira2010learning} for the Ornstein-Uhlenbeck case and \cite{periera2014support} for the non-linear diffusion case.

\section{Estimating the Support}
\label{sec:Estimation}

\subsection{Empirical Contrast: the Commutator}
\label{sec:commutator}
The methodology presented in the paper relies on the fact that the target matrix $\W$ commutes with the matrix $\mathsf{K}$. Indeed, recall that $\mathsf{K}:=f(\W)$, see Section~\ref{sec:Model} for a definition of this notation. In particular, there exist an orthonormal matrix $U$ and a diagonal matrix~$D$ with diagonal entries $(\lambda_1,\ldots,\lambda_1,\ldots,\lambda_s,\ldots,\lambda_s)$ with multiplicities $(\ell_1,\ldots,\ell_s)$ such that $\W=UDU^\top$ and $\mathsf{K}=Uf(D)U^\top$ where $f(D)$ is a diagonal matrix with diagonal entries $(f(\lambda_1),\ldots,f(\lambda_1),\ldots,f(\lambda_s),\ldots,f(\lambda_s))$ and same multiplicities as above. Since $f$ is assumed injective (and hence one to one on the spectrum of $\W$), the matrices~$\W$ and $\mathsf{K}$ have exactly the same eigenspaces in the sense that the eigenspace~$E_{\lambda_k}(\W)$ associated to~$\lambda_k$ (in the decomposition of $\W$) is exactly the one associated to $f(\lambda_k)$ (in the decomposition of $\mathsf{K}$), namely 
\begin{equation}
\label{eq:eigenspace_equality}
E_{f(\lambda_k)}(\mathsf{K})=E_{\lambda_k}(\W)
\end{equation}
and the dimension of this eigenspace is $\ell_k$, the multiplicity of $\lambda_k$. It follows that, when $F$-identifiability holds, the only solutions $\mathsf{A}$ with $\supp(\mathsf{A}) \subseteq \overline F$ to $ \mathsf{A} \mathsf{K} = \mathsf{K} \mathsf{A}$ are of the form $\mathsf{A} =t \mathsf{W}  $ for some $t \in \mathbb R$. 

\begin{remark}[Reminder on matrix perturbation theory]
Now, we do not observe $\mathsf{K}$ but a noisy version $\widehat{\mathsf{K}}$. For instance, $\widehat{\mathsf{K}}$ is the estimation of~$\mathsf{K}$ from a finite sample. The nice decomposition~\eqref{eq:eigenspace_equality} does not hold anymore changing $\mathsf{K}$ by $\widehat{\mathsf{K}}$. But there exist an orthonormal matrix $\widehat U$, a diagonal matrix $\widehat D$  with diagonal entries $(\widehat \mu_1,\ldots,\widehat\mu_N)$ such that $\widehat{\mathsf{K}}=\widehat U\widehat D\widehat U^\top$ and the following holds. Mirsky's inequality \citep[Corollary~4.12]{stewart1990matrix} and the Wedin'\! $\sin(\theta)$ theorem \citep[P.~260]{stewart1990matrix} show that, for~$\widehat{\mathsf{K}}$ such that $\|\mathsf{K}-\widehat{\mathsf{K}}\|$ is sufficiently small (with respect to the minimal separation $|f(\lambda_{k_1})-f(\lambda_{k_2})|$ between distinct eigenvalues), then for all $k=1,\ldots,s$, the eigenvalues $\widehat\mu_{(\sum_{t=1}^{k-1}\ell_{t})+1},\ldots,\widehat\mu_{\sum_{t=1}^{k}\ell_{t}}$ are close to $f(\lambda_k)$ and the space spanned by a group of eigenvectors, namely the vectors $\widehat U_{(\sum_{t=1}^{k-1}\ell_{t})+1},\ldots,\widehat U_{\sum_{t=1}^{k}\ell_{t}}$, is close to $E_{f(\lambda_k)}(\mathsf{K})=E_{\lambda_k}(\W)$ (more precisely, the orthonormal projections onto these spaces are close in Frobenius norm). 

If we consider $A$ such that $A\widehat{\mathsf{K}}=\widehat{\mathsf{K}}A $ then again these matrices share the same eigenspaces and we conclude that, up to label switching, the eigenvectors $(V_k)_{k=1}^N$ of $A$ are such that the spaces spanned by the group of eigenvectors $V_{(\sum_{t=1}^{k-1}\ell_{t})+1},\ldots,V_{\sum_{t=1}^{k}\ell_{t}}$ are close to the targets~$E_{\lambda_k}(\W)$, for $k=1,\ldots,s$.
\end{remark}

However, the choice $A=\W$ does not satisfy $A\widehat{\mathsf{K}}=\widehat{\mathsf{K}}A $ and we need to relax this identity. Furthermore, remark that  $\W\widehat{\mathsf{K}}-\widehat{\mathsf{K}}\W=\W \mathsf E-\mathsf E\W$ denoting $\mathsf E=\widehat{\mathsf{K}}-\K$ the estimation errors. It follows 
\eq
\label{eq:nearlycommuteW}
\frac{\Vert  \mathsf{W} \widehat{\mathsf{K}}-\widehat{\mathsf{K}} \mathsf{W} \Vert}{\lVert \mathsf{W}\lVert}=\frac{\Vert  \mathsf{W} {\mathsf{E}}-{\mathsf{E}} \mathsf{W} \Vert}{\lVert \mathsf{W}\lVert}\leq2\lVert \mathsf{E}\lVert\,.
\qe
In view of \eqref{eq:nearlycommuteW} and of the discussion above, we use the following cost function
$$ 
\mathsf{A} \mapsto \frac{\Vert  \mathsf{A} \widehat{\mathsf{K}}-\widehat{\mathsf{K}} \mathsf{A} \Vert}{\lVert \mathsf{A}\lVert},\quad \mathsf{A}\in\mathcal E( \overline F) \setminus \{0 \},$$
which aims at matrices for which the spaces spanned by some groups of it eigenspaces are close to the eigenspaces of the target. This empirical criterion was first used in \cite{fytp14} in a similar context to reflect that $\W$ is expected to nearly commute with $\widehat{\mathsf{K}}$, provided that~$\widehat{\mathsf{K}}$ is sufficiently close to its true value $\mathsf{K}$, see for instance~\eqref{eq:nearlycommuteW}.

\subsection{The $\ell_0$-approach}
\label{sec:L0approach}
Given an estimator $\widehat{\mathsf{K}}=\widehat{\mathsf{K}}_n$ of $\mathsf{K}$ build from a sample of size $n$ and a set of forbidden entries~$F$ reflecting \eqref{Hyp:Forbid}, we construct an estimator $\widehat S=\widehat S_n$ of the target support $S^\star:= \supp(\W)$ as a minimizer of the criterion $Q$ given by
\[
\forall S\subseteq \overline F,\quad
Q(S) := \min_{\mathsf{A}\in\mathcal E(S) \setminus \{0 \}} \frac{\Vert  \mathsf{A} \widehat{\mathsf{K}}-\widehat{\mathsf{K}} \mathsf{A} \Vert}{\lVert \mathsf{A}\lVert} + \lambda_n \vert S \vert,
\]
for some tuning parameter $\lambda_n >0$ and defining the minimum of an empty set as $\infty$. Recall that~$\mathcal E(S)$ is the set of symmetric matrices $\mathsf{A}$ such that $\operatorname{Supp}(\mathsf{A})\subseteq S$. Our estimator is 
\[
\widehat S\in\arg\min_{S\subseteq \overline F}Q(S)
\]
Furthermore, we assume that the estimator $\widehat{\mathsf{K}}$ converges toward $\mathsf{K}$ in probability $R_n$, namely
\eq
\label{a:conc}
\tag{$\mathrm{\mathbf{H}_2}$}
\forall t>0\,,\quad\mathbb{P}\big\{\Vert \widehat{\mathsf{K}}-\mathsf{K} \Vert \geq t  \big\} \leq R_n(t),
\qe
where $t\mapsto R_n(t)$ is non-increasing and such that, for all $t>0$, $R_n(t)\to0$ as $n$ goes to $\infty$.
\noindent

\begin{theorem}
\label{t:main1}
Assume that \eqref{a:conc} and \eqref{Hyp:Forbid} hold. If $\W$ is $F$-identifiable, then
\[
\mathbb{P}\big\{ \widehat{S} \neq S^\star \big\} \leq   R_n\Big(\frac{c_0(S^\star)-\lambda_n \vert S^\star \vert}{4}\Big) +R_n\Big( \frac{\lambda_n}{2} \Big)\,,
\]
where
\eq
\label{eq:def_c0}
c_0(S^\star) := \min_{\substack{ S \neq S^\star \\ \vert S \vert \leq \vert S^\star \vert}} \min_{ \mathsf{A} \in \mathcal E(S)} \frac{\Vert \mathsf{A}\mathsf{K}-\mathsf{K}\mathsf{A} \Vert}{\lVert \mathsf{A}\lVert}>0\,.
\qe
\end{theorem}

\noindent A proof of Theorem \ref{t:main1} is given in Section \ref{proof:ThmL0}.

\begin{corollary}\label{c:main1}
Under the assumptions of Theorem \ref{t:main1}, if it holds
\[
\lambda_n \rightarrow 0\quad \mathrm{and}\quad \sum_{n \in \mathbb N} R_n\Big( \frac{\lambda_n}{2} \Big)< +\infty\,,
\]
then one has $\widehat{S} \rightarrow S^\star$ almost surely.
\end{corollary}

\noindent Note that, based on the upper bound in Theorem \ref{t:main1}, a good scaling may be $\lambda_n^\star= \frac{c_0(S^\star)}{\vert S^\star \vert +4}$ leading to the upper bound 
\[
\mathbb{P}\big\{ \widehat{S} \neq S^\star \big\} \leq  2 R_n\Big(\frac{c_0(S^\star)}{2\vert S^\star \vert +8}\Big)\xrightarrow{n\to\infty}0
\] 
which is optimal up to a constant less than $2$. Interestingly, this oracle choice $\lambda_n^\star$ does not depend on $n$ but this calibration is not relevant since both $c_0(S^\star)$ and $\vert S^\star \vert$ are unknown. Alternatively, we may choose a sequence $\lambda_n$ decreasing slowly to $0$ to ensure both conditions of Corollary \ref{c:main1}.

\subsection{Edge significance based on the commutator criterion}
\label{subsec:backward}
The $\ell_0$-approach meets with the curse of dimensionality. In practice, a backward methodology provides a computationally feasible alternative to the support reconstruction problem. Starting from the maximal acceptable support $\overline F$, the idea of the backward procedure is to remove the least significant entries one at a time and stop when every entry is significant. Using the corresponding small case letter to denote the vectorization of a matrix, {\it e.g.}, $a=\vect(\mathsf{A}) =  (\mathsf{A}_{11},...,\mathsf{A}_{N1},...,\mathsf{A}_{1N},..., \mathsf{A}_{NN})^\top$, significancy can be leveraged using the Frobenius norm of the commutator operator $a \mapsto \Delta(\mathsf{K}) a = \vect(\mathsf{K} \mathsf{A} - \mathsf{A} \mathsf{K})$, where
$$ \Delta(\mathsf{K}) = \id \otimes \mathsf{K} - \mathsf{K} \otimes \id  \ \ \in \mathbb R^{N^2 \times N^2}  $$
and $\otimes$ denotes the Kronecker product. Indeed, searching for the target $\W$ in the commutant of~$\mathsf{K}$ reduces to searching for $w = \vect(\W)$ in $\ker ( \Delta(\mathsf{K}))$, the kernel of $\Delta(\mathsf{K})$. Because the Frobenius norm coincides with the Euclidean norm of the vectorization, the functions $\mathsf{A} \mapsto \Vert \widehat{\mathsf{K}} \mathsf{A} - \mathsf{A} \widehat{\mathsf{K}} \Vert^2$ and $a \mapsto \Vert \Delta(\widehat{\mathsf{K}}) a \Vert^2$ can be used indistinctly as cost functions. Minimizing this criterion over model spaces of decreasing size, we may consider sequences of least-squares estimates in the sequel.

\subsubsection*{Assumptions}
Assume the three following hypotheses \eqref{e:HypSigma}, \eqref{e:Hyp1} and \eqref{eq:id}.

$\circ$ Deriving the asymptotic law of least-squares estimators, we may assume that the estimate~$\widehat{\mathsf{K}}$ is such that 
\begin{equation}
\label{e:HypSigma}
\tag{$\mathrm{\mathbf{H}_\Sigma}$}
\sqrt n (\widehat k - k) \xrightarrow[n \to \infty]{d} \mathcal N(0,\Sigma),   
\end{equation}
where $\Sigma$ is a $N^2 \times N^2$ covariance matrix (either known or that can be estimated). For instance, one can think of $\widehat{\mathsf{K}}$ as the empirical covariance when observing a sample of vectors of covariance~${\mathsf{K}}$. This condition is verified for instance in the framework considered in \cite{fytp14,Barsotti2016}. Note that asymptotic normality is a standard ground base investigating any least-squares procedure.

$\circ$ In order to exclude the trivial solution $a=0$, the target $\W$ is assumed normalized
\begin{equation}
\label{e:Hyp1}
\tag{$\mathrm{\mathbf{H}_1}$}
\mathbf 1^\top w =1,   
\end{equation}
where $\mathbf 1$ has all its entries equal to one. Because the available information on $\W$ is of spectral nature and as such, is scale-invariant, a normalization of some kind is crucial for the identifiability. Here, the condition $\mathbf 1^\top w =1$ achieves two goals: preventing the null matrix form being a solution and making the problem identifiable. 
\begin{remark}
The main drawback of this normalization concerns the situation where the entries of $\W$ sum up to zero, in which case the normalization is impossible. If the context suggests that the solution may be such that $\mathbf 1^\top w =0$, a different affine normalization $\mathbf v^\top w=1$ (with any fixed vector $\mathbf v$) must be used, without major changes in the methodology. In practice, one may consider the vector~$\mathbf v$ at random (for instance with isotropic law), so that \eqref{e:Hyp1} is almost surely fulfilled for any fixed target $w$.

Observe that if one knows in advance that the target $\W$ has nonnegative entries then the normalization~\eqref{e:Hyp1} is acceptable.
\end{remark}

$\circ$ For $S $ a support included in $\overline F$, we aim at a solution in the affine space 
$$ \mathcal A_S := \{ a=\vect(\mathsf{A}): \supp(\mathsf{A}) \subseteq S, \ \mathsf{A}=\mathsf{A}^\top \ , \ \mathbf 1^\top a = 1 \}. $$ 
with  linear difference space given by
$$\mathcal L_{S} := \{ a=\vect(\mathsf{A}): \supp(\mathsf{A}) \subseteq S, \ \mathsf{A}=\mathsf{A}^\top \ , \ \mathbf 1^\top a = 0 \}\,.$$
By abuse of notation, $\mathcal A_S$ may refer both to the space of matrices or their vectorizations. To find the target support $S^\star$\!, one must exploit the fact that the vector $w$ lies in the intersection of $\ker( \Delta(\mathsf{K}))$ and $\mathcal A_{\overline F}$. Actually, $w$ can then be recovered if the intersection is reduced to the singleton $\{w \}$. In this case, the matrix $\W$ and its support $S^\star $ are $F$-identifiable. Hence, we assume that
\begin{equation}\label{eq:id}
\tag{$\mathrm{\mathbf{H}_{Id}}$}
 \ker ( \Delta(\mathsf{K}) ) \cap \mathcal L_{\overline F} = \{0 \},
 \end{equation}
 which is implied by $F$-identifiability, see Definition \ref{def:Fidentifiability}.

 \subsubsection*{Asymptotic normality and a significance test}
The framework under consideration can be viewed as a heteroscedastic linear regression model with noisy design for which $w = \vect(\W)$ is the parameter of interest. Indeed, consider for each support $S \subseteq \overline F$ a full-ranked matrix $\Phi_S \in \mathbb R^{N^2 \times \dim(\mathcal A_S)}$ whose column vectors form a basis of~$\mathcal L_S$. Assuming that $\W$ is $F$-identifiable and taking $S \subseteq \overline F$, the operator $\Delta(\mathsf{K}) \Phi_S$ is one-to-one. In this case, evaluating the commutator $a \mapsto \Delta(\mathsf{K}) a$ over $\mathcal A_S$ reduces to considering the map
$$ b \mapsto \Delta(\mathsf{K}) (a_0 - \Phi_S b) \,, \quad b \in \mathbb R^{\dim(\mathcal A_S)},    $$
with $a_0$ chosen arbitrarily in $\mathcal A_S$. When replacing the unknown $\Delta(\mathsf{K})$ with its estimate~$\Delta(\widehat{\mathsf{K}})$, the minimization of the criterion $a \mapsto \Vert \Delta(\widehat{\mathsf{K}}) a \Vert^2$ over $\mathcal A_S$ can be written similarly as a linear regression framework where the parameter of interest is estimated by
\begin{equation} \label{hatbetaS} \widehat \beta_S \in \arg \min_{b \in \mathbb R^{\dim(\mathcal A_S)}} \Vert \Delta(\widehat{\mathsf{K}}) (a_0 - \Phi_S b)  \Vert^2. 
\end{equation}
We recognize a linear model with response $y = \Delta(\widehat{\mathsf{K}}) a_0$ and noisy design matrix~$X = \Delta(\widehat{\mathsf{K}})\Phi_S$. In this setting, remark that $w = a_0 - \Phi_S \beta$ with $\beta$ the unique solution to~$\Delta(\mathsf{K}) ( a_0 - \Phi_S \beta) = 0$. Denoting by $\mathsf{M}^\dagger$ the pseudo-inverse of a matrix $\mathsf{M}$, we deduce the following result.

\begin{theorem}
\label{asymptbeta} If $S^\star \subseteq S$, the estimator $\widehat \beta_S$ is asymptotically Gaussian with
$$  \sqrt n (\widehat \beta_S - \beta) \xrightarrow[n \to \infty]{d} \mathcal N (0, \Omega_S ),    $$
where $\Omega_S := \big(\Phi_S^\top \Delta(\mathsf{K})\big)^\dagger \Delta(\mathsf{W}) \Sigma \Delta(\mathsf{W}) \big(\Delta(\mathsf{K}) \Phi_S \big)^\dagger$.
\end{theorem}
We then have
\begin{equation}
\label{eq:ws}
 \widehat w_S = \vect(\widehat{\mathsf{W}}_S)=  \arg \min_{a \in \mathcal A_S} \Vert \Delta(\widehat{\mathsf{K}}) a \Vert^2 = a_0 - \Phi_S \widehat \beta_S \,.
 \end{equation}
The asymptotic distribution of $\widehat w_S$ follows directly from Theorem \ref{asymptbeta}, 
\begin{equation}\label{norm_w} \sqrt n (\widehat w_s - w) \xrightarrow[n \to \infty]{d} \mathcal N \big(0, \Phi_S \Omega_S \Phi_S^\top \big).  \end{equation}
The limit covariance matrix is unknown, but plugging the estimates $\widehat{\mathsf{W}}_S$, $\widehat{\mathsf{K}}$ and $\widehat \Sigma$ yields an estimator $\Phi_S \widehat \Omega_S \Phi_S^\top$, which is consistent under the $F$-identifiability assumption. In particular, the diagonal entry of $\Phi_S \widehat \Omega_S \Phi_S^\top$ associated to the $(i,j)$-entry of $\mathsf W$, which we denote $\widehat \sigma^2_{S,ij}$, provides a consistent estimator for the asymptotic variance of $\widehat{\mathsf{W}}_{S,ij}$. As a result, the statistic
\begin{equation}
\label{e:test_stat}
 \tau_{ij}(S) := \sqrt n \frac{\widehat{\mathsf{W}}_{S,ij}}{\widehat \sigma_{S,ij}}   
\end{equation}
can be used to measure the relative significance of the estimated entry $\widehat{\mathsf{W}}_{S,ij}$. The backward support selection procedure is then implemented by the recursive algorithm as follows. 

\bigskip 

\begin{algorithm}[H]
\caption{Backward algorithm for support selection}
  \SetAlgoLined
  \KwData{A set of forbidden entries $F$, a matrix $\widehat{\K}$.}
  \KwResult{A sequence of estimators $\widehat \W_{S_1},\widehat \W_{S_2},...$ with nested supports $S_1 \supset S_2 \supset ...$.}
    \BlankLine
    \begin{algorithmic}[1]
  \STATE Start with the maximal acceptable support $S_1=\overline F$,
  \STATE At each step $k$, compute the statistics $\tau_{ij}(S_k)$ for all $(i,j) \in S_k$,
  \STATE Remove the least significant edge $(i,j)$ which minimizes $| \tau_{ij}(S_k)| $ for $ (i,j) \in S_k$, and set $S_{k+1} = S_k \setminus \{ (i,j), (j,i) \}$,
  \STATE Stop when all edges have been removed.	
  \end{algorithmic}
\label{alg:BackwardSupportSelection}
\end{algorithm}

\bigskip

The backward algorithm produces a sequence of nested supports that one can choose to stop once all the edges are judged significant, that is, when all the statistics $\tau_{ij}(S_k)$, $(i,j) \in S_k$ exceed in absolute value some fixed threshold $\tau_0$. Owing to the asymptotic normality of $\widehat{\mathsf{W}}_{S,ij}$ shown in Eq.\! \eqref{norm_w}, the $(1-\frac \alpha 2)$-quantile of the standard Gaussian distribution would appear as a reasonable  choice for the threshold $\tau_0$, as it boils down to performing an asymptotic significance test of level $\alpha$. However, due to the slow convergence to the limit distribution and the tendency to overestimate the variance for small sample sizes (see Figure~\ref{fig:conv}), a threshold based on the Gaussian quantile inevitably  leads to an overly large estimated support. Nevertheless, we show that an adaptive calibration of the threshold can be achieved by considering the overall behavior of the commutator~$\Delta(\widehat{\mathsf{K}}) \widehat w_{S_m}$ computed over the nested sequence of active supports. 

\begin{figure}[h]
 \centering
 \includegraphics[width = 11cm,height = 5cm]{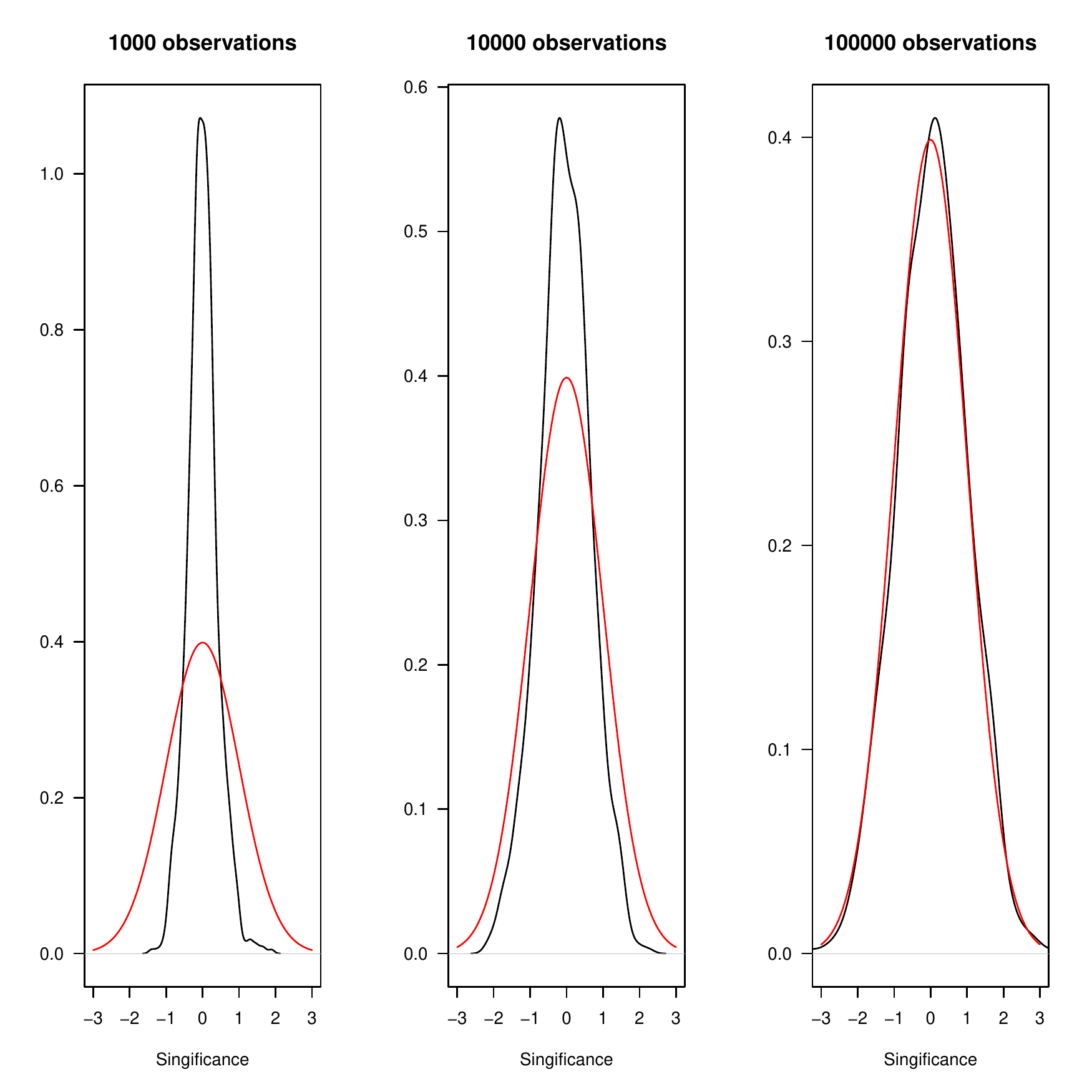}
 \caption{Estimated density of the statistic $\tau_{ij}(S)$ for an edge $(i,j) \in S \setminus S^\star$ compared to its theoretical Gaussian limit distribution, for samples of size $n=1000$ (left), $n=10000$ (center) and $n=100000$ (right).}
 \label{fig:conv}
\end{figure}


\subsubsection*{Calibration of the threshold by cross-validation}

By removing the least significant edge at each step, the backward algorithm generates a sequence of nested active supports $S_1 \supset \cdots \supset S_\ell$, that we refer to as a ``trajectory''. Along this trajectory, we compute the empirical contrast defined by
\eq
\label{e:Critere_Normalise}
\forall  S\subseteq\overline F  ,\quad
S \mapsto \Crit(\widehat{\mathsf{W}}_S,\widehat\K):=\frac{\Vert  \widehat{\mathsf{W}}_S \widehat{\mathsf{K}}-\widehat{\mathsf{K}}  \widehat{\mathsf{W}}_S \Vert}{\lVert  \widehat{\mathsf{W}}_S \lVert}\,.
\qe
Note that computing this criterion boils down to compute $\widehat{\mathsf{W}}_S$ which is a simple projection onto~$\mathcal A_S$ as shown in \eqref{eq:ws}.

When the true support $S^\star$ lies in the trajectory, one expects to observe a ``gap'' in the sequence $j \mapsto \Crit(\widehat{\mathsf{W}}_{S_j},\widehat\K)$
when $S_j$ goes from $S^\star$ to a smaller support. Indeed:
\begin{itemize}
	\item For $S^\star \subseteq S$, the target $\W$ is consistently estimated by $\widehat \W_S$ so that $\Crit(\widehat{\mathsf{W}}_S,\widehat\K)$ tends to zero at rate~$\sqrt n$,
	\item For $S\subsetneq S^\star $, the lower bound $\Vert \mathsf A \widehat{\mathsf{K}} - \widehat{\mathsf{K}}  \mathsf A  \Vert \geq \Vert \mathsf A \mathsf{K} -\mathsf{K}  \mathsf A  \Vert - 2 \Vert \widehat{ \mathsf{K}} - \mathsf{K} \Vert \Vert \mathsf A \Vert$ yields
\begin{align} \label{eq:threshold}
\Crit(\widehat{\mathsf{W}}_S,\widehat\K) = \frac{\Vert  \widehat{\mathsf{W}}_S \widehat{\mathsf{K}}-\widehat{\mathsf{K}}  \widehat{\mathsf{W}}_S \Vert}{\lVert  \widehat{\mathsf{W}}_S \lVert} 
\geq c(S)  - 2 \Vert \widehat{\K} - \K \Vert
\end{align}
with $ c(S) : = \min_{A \in \mathcal A_S} \Vert  \mathsf A \mathsf{K} -\mathsf{K}\mathsf A  \Vert/\lVert \mathsf A \lVert$ a positive constant. In particular, one has 
\[
\min_{S\subsetneq S^\star }  c(S)\geq
\min_{\substack{ S \neq S^\star \\ \vert S \vert \leq \vert S^\star \vert}}  c(S)=c_0(S^\star)>0
\]
where $c_0(S^\star)$ is defined in~\eqref{eq:def_c0}.
\end{itemize}

 {In some way, $c_0(S^\star)$ measures the amplitude of the signal: one expects to be able to recover the target $\W$ when the  estimation error $\Vert \widehat \K - \K \Vert$ reaches at least the same order as $c_0(S^\star)$. The true support~$S^\star$ then }corresponds to a transitional gap in the contrast curve that can be captured by a suitably chosen threshold~$t>0$. Since $\widehat \K$ converges toward $\K$ in probability, any threshold $0< t < c_0(S^\star)$ will work with probability one asymptotically.

\begin{remark} The condition that $S^\star$ lies in the trajectory of nested supports is crucial to detect the commutation gap, although seldom verified in practice due to the tremendous amount of testable supports. This issue is specifically targeted by the Bagging version of the backward algorithm discussed in Section \ref{subsec:graphtri}. 
\end{remark}

An obstacle  to the detection of the commutation gap is the increasing behavior of the commutator over the nested trajectory $S_1 \supset \cdots \supset S_\ell$. This phenomenon, indirectly caused by the dependence between the trajectory and $\widehat \K$, can be annihilated when considering the empirical contrast over a trajectory built from a training sample. In fact, the monotonicity can even be ``reversed'' before reaching the true support if the $\widehat \W_{S_j}$ are estimated independently from $\widehat \K$. 
This can be explained as follows: Consider the ideal scenario where estimators $\widetilde \W_{S_1},...,\widetilde \W_{S_\ell}$ are built from the backward algorithm using an estimator $\widetilde \K$ independent from $\widehat \K$. We assume moreover that the true support $S^\star$ lies in the trajectory $S_1 \supset ... \supset S_\ell$. The trick is to write
$$ \Delta(\widehat \K) \widetilde w_{s_j}  =\Delta(\widetilde \K) w + \Delta(\K) \widetilde w_{S_j}  + \Delta(\widehat \K - \K)  (\widetilde w_{S_j} - w), $$
and to analyze the three terms separately:
\begin{itemize}
	\item The term $\Delta(\widehat \K) w $ has no influence as it is  common to all supports in the trajectory.
	\item The term $ \Delta(\K) \widetilde w_{S_j} $ approaches zero as~$ \widetilde w_{S_j} $ gets closer to $w$. Heuristically, the variance of~$\widetilde w_{S_j}$, and incidentally that of $ \Delta(\K) \widetilde w_{S_j} $, is larger for over-fitting supports  $ S \supseteq S^\star$. This results in  the sequence $j \mapsto \Delta(\K) \widetilde w_{S_j} $ being stochastically decreasing as $S_j$ approaches~$S^\star$ from above. On the other hand, the bias is expected to dominate  once the trajectory passes through the true value $S^\star$, making the remaining of the sequence $ \Delta(\K) \widetilde w_{S_j} $increase stochastically.
	\item The term $\Delta(\widehat \K - \K)  (\widetilde w_{S_j} - w)$ is negligible for $S \supseteq S^\star$, as both $\widehat \K - \K$ and $\widetilde w_{S_j} - w$ tend to zero independently. We emphasize that this argument no longer holds without the independence of $\widetilde w_{S_j} $ and $\widehat \K$. This is precisely why we use a training sample.
\end{itemize}

\noindent
Thus, the sequence $ j\mapsto \Crit(\widetilde \W_{S_{j}},\widehat{\K})  = \Vert \Delta(\widehat \K) \widetilde w_{S_j} \Vert/\Vert \widetilde w_{S_j}\Vert $ is expected to achieve its minimum for the best estimator $\widetilde w_{S_j}$ in the trajectory, that is for $S_j = S^\star$. Furthermore, beyond the true support (for small active supports), $\widetilde w_{S_j}$ is not a consistent estimator of $w$ so that the criterion~no longer approaches zero, resulting in the so-called commutation gap. 

The ``reversed'' monotonicity provides an easy way to calibrate the threshold in the backward algorithm. Indeed, since $S_j \mapsto \Delta(\widehat \K) \widetilde w_{S_j} $ is expected to decrease when approaching the true support (coming from larger active supports along a trajectory), the estimated support can be heuristically chosen has the last time the criterion is below an adaptive threshold, see Figure~\ref{fig:Errorboot}. In particular, $\Crit(\widetilde \W_{S_{1}},\widehat{\K})$ can be used as an adaptive threshold for the backward algorithm when the estimator $\widehat \K$ and the trajectory $S_1 \supset \cdots \supset S_\ell$ are obtained from independent samples.

Of course, to afford splitting the sample to build the $\widetilde \W_{S_j}$ independent from $\widehat \K$ may be unrealistic. Nevertheless, the numerical study suggests that the independence is well mimicked when~$\widehat{\mathsf K}$ is built from the whole dataset but the backward algorithm sequence $\widetilde \W_{S_1},...,\widetilde \W_{S_\ell}$ is obtained from a learning sub-sample, as illustrated in Figure \ref{fig:Errorboot}. {Empirically, the optimal size of training samples could be calibrated in function of the number of observations using the robustness of the outputs of the algorithm. In this paper, we always draw training samples by taking each observation with probability $1/2$, with no consideration regarding the size of the whole sample.}\vspace{-0.4cm}

\begin{figure}[h]
 \centering
 \includegraphics[scale = 0.4]{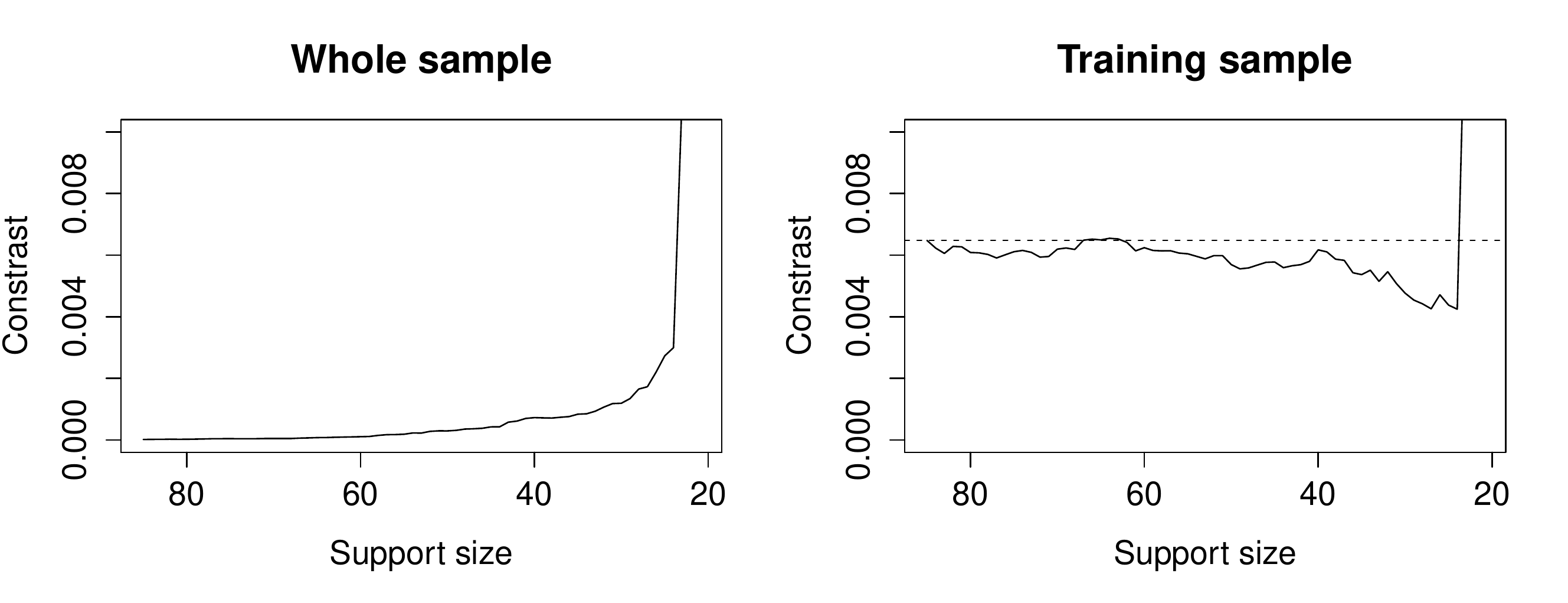}
 \caption{The contrast sequence $j \mapsto \Crit(\widetilde \W_{S_{j}},\widehat{\K}) $ computed in the example of Section \ref{sec:simu2}. The nested support sequence and estimators $\widetilde \W_{S_{j}}$ are obtained from the backward algorithm implemented on the whole sample (left) and on a training sample of half size (right). In both cases, $\widehat{\mathsf K}$ is constructed from the whole sample. Using a training sample manages to reverse the monotonicity in the first part of the sequence, thus making the commutation gap easier to locate. The initial value of the sequence $t=\Crit(\widetilde \W_{S_{1}},\widehat{\K}) $ then provides a tractable adaptive choice for the threshold.}
 \label{fig:Errorboot}
\end{figure}

\subsection{Improving the backward algorithm by Bagging}
\label{subsec:graphtri}
The main weakness of the backward procedure remains that it requires the true support $S^\star$ to lie in the trajectory $S_1 \supset \cdots \supset S_\ell$ obtained from removing the least significant edge one at a time. In practice, this condition is rarely verified, especially with small datasets. A way to solve this issue is to replicate the backward algorithm over a collection of random sub-samples, a process commonly known to as \textit{bagging}\textemdash  Bootstrap Aggregation. The description of this algorithm is given in Algorithm~\ref{alg:Backward}.

\begin{algorithm}[h]
\caption{Bagging backward algorithm}
  \SetAlgoLined
  \KwData{A set of forbidden entries $F$, a sample $X$.}
  \KwResult{A collection of estimated supports $\widehat S_m, m =1,...,M$.}
    \BlankLine
    \begin{algorithmic}[1]
  	\STATE Build $M$ bootstrapped samples without replacement.
		\STATE For each sub-sample $m=1,...,M$, build an estimator $\widetilde \K_m$ of $\K$.
  	\STATE For all $m$, run Algorithm \ref{alg:BackwardSupportSelection} without stopping condition and return $M$ trajectories $S_{1m} \supset \cdots \supset S_{\ell m}$ and the corresponding estimators $\widetilde \W_{S_{km}}$.
		\STATE Evaluate the empirical contrast $\Crit(\widetilde \W_{S_{km}},\widehat{\K})$ over each trajectory with the estimator $\widehat \K$ calculated from the whole sample.
		\STATE For each trajectory, return the estimated support $\widehat S_m : = S_{\hat k_m m}$ as the last support whose contrast lies below the initial value:	
		$$ \hat k_m := \max \big\{k =1,...,\ell : \Crit(\widetilde \W_{S_{km}},\widehat{\K}) \leq \Crit(\widetilde \W_{S_{1m}},\widehat{\K}) \big\}.		\vspace{-0.2cm} $$		
\end{algorithmic}
\label{alg:Backward}
\end{algorithm}

The bagging algorithm produces a collection of estimated supports in a way to make the final decision more robust. At this point, several solutions are possible: select the most represented support among the $\widehat S_m$'s, keep the edges present in the most supports etc... A preliminary detection of the outliers among the $\widehat S_m$'s, e.g. by removing beforehand the supports $\widehat S_m$'s that are either too big or too small, might also considerably improve the method, as we illustrate on actual examples in Section~\ref{sec:num}.

\section{Numerical study}\label{sec:num}
\subsection{Toy example}\label{subsec:toy}
In the previous section, we have introduced different algorithms. To emphasize the motivation of the bagging algorithm, we consider a simple example, and implement the different algorithms for support recovery. To check the performances of the $\ell_0$ procedure, we need to consider a graph with a small number of vertices (since the $\ell_0$ complexity grows with $2^{N(N-1)/2}$ where $N$ denotes the number of vertices). Here, we consider the graph $G_1$ represented in Figure~\ref{fig:kite5}, the kite graph on $5$ vertices.

\begin{figure}[H]
\center
\begin{tikzpicture}[line width =0.3mm,scale=0.6]
\vertex[circle,  minimum size=10pt](1) at  (0,0){} ;
\vertex[circle,  minimum size=10pt](2) at  (2.6,0) {};
\vertex[circle,  minimum size=10pt](5) at (5.2,0){};
\vertex[circle,  minimum size=10pt](6) at (7.8,1.3){};
\vertex[circle,  minimum size=10pt](7) at (7.8,-1.3){};

\draw (1)--(2)--(5)--(6)--(7)--(5) ;
\end{tikzpicture}
\caption{The kite graph $G_1= \nabla_5$.}
\label{fig:kite5}
\end{figure}
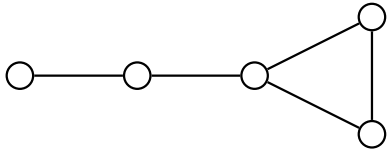

We choose $\W$ as the (normalized) adjacency matrix of $G_1$ then draw a sample of size $n = 500$ of centered Gaussian vectors $X_1, \cdots X_{n}$ of $\R^5$ with covariance matrix $\K = \exp(\W)$. We assume known that $G_1$ contains no self-loop so that we take $F = F_{diag}$ as the set of forbidden values. {In this simple example, the constant $c_0(S^\star)$ (see Eq.\! \eqref{eq:def_c0}) can be calculated explicitly, yielding $c_0(S^\star) \approx 0.12$. In comparison,  for $n=500$, $\mathbb E \Vert \widehat \K - \K \Vert$ is evaluated to approximately $0.27$ by Monte-Carlo. We expect to be able to recover the true support when the noise level drops below the signal amplitude. Based on the bound of Eq.\! \eqref{eq:threshold}, this occurs  as soon as $\Vert \widehat \K - \K \Vert \leq c_0(S^\star)/2$ however, because this bound is not sharp, a lesser level of precision is required in practice.}\\

\noindent We compare the following algorithms:
\begin{enumerate}
 \item \textit{Contrast penalized $\ell_0$ minimization with optimal penalization constant.} We compute
$$  \widehat S = \arg \min_{S \subseteq \overline F_{diag}} \bigg\{ \min_{\mathsf{A}\in\mathcal E(S) \setminus \{0 \}} \frac{\Vert  \mathsf{A} \widehat{\mathsf{K}}-\widehat{\mathsf{K}} \mathsf{A} \Vert}{\lVert \mathsf{A}\lVert} + \lambda \vert S \vert \bigg\}. $$
The constant $\lambda$ is chosen as  the best possible value, minimizing the {oracle} error $\delta(\widehat S) $ measured by the symmetric difference between $\widehat S$ and $S^\star$: $\delta(\widehat S) = |  \widehat S \cup S^\star \setminus \widehat S \cap S^\star |$. Because the calibration parameter $\lambda$ is chosen optimally for each realization of $\widehat \K$, the numerical performances of the method can be expected to be overestimated compared to a fully data-driven procedure. 
 \item \textit{Thresholding contrast minimization with optimal threshold. }The target matrix $\W$ is estimated over the maximal acceptable support $\overline F_{diag}$. We then compute 
$$ \widehat S = \{ (i,j):  | \widehat \W_{ij} | > t \}, $$
where the threshold $t$ is chosen so as to minimize the {oracle} error $\delta(\widehat S)$ for each realization of $\widehat \K$.
 \item \textit{Backward algorithm.}  We generate a training sample by taking each observation with probability $1/2$ independently, from the whole sample. The estimator of $\K$ in this sub-sample is denoted $\widetilde \K$. We implement  Algorithm \ref{alg:BackwardSupportSelection} on $\widetilde \K$, yielding a trajectory $S_{1} \supset ...  \supset S_{\ell }$ of nested supports whose sizes vary from $|S_{1}| = 20$ (the  full off-diagonal support) to $|S_{\ell}| = 12$ (the minimal size required for diagonal identifiability), along with the associated estimators $\widetilde \W_{S_{k}}, k=1,...,\ell$. Remark that because the supports are symmetric, two entries are removed at each step so that $\ell = 5$ in this case. We then compute the threshold $t = \Crit(\widetilde \W_{S_{1}},\widehat{\K}) $ corresponding to the initial value of the contrast. The estimated support~$\widehat S$ is defined as the smallest support $S$ in the trajectory such that $\Crit(\widehat \W_{S},\widehat{\K})  \leq t $.

 \item \textit{Bagging backward algorithm.} The previous algorithm is implemented over $M=100$ training samples drawn keeping observations with probability $1/2$. For each $m=1,...,M$, we retain
\begin{itemize}
	\item[-] the threshold $t_m = \Crit(\widehat \W_{S_{1m}},\widehat{\K}) $ corresponding to the initial value of the contrast,
	\item[-] the estimated support, that is, the smallest support $\widehat S_m$ in the trajectory such that $\Crit(\widehat \W_{\widehat S_m},\widehat{\K})  \leq t_m $.
\end{itemize}
The final decision $\widehat S$ is obtained as follows. Only a proportion $q$ of the training samples $m$ with a small initial contrast $t_m$, which are expected to provide more accurate results, are kept (in the whole paper, we chose $q = 2/\sqrt{M}$ empirically). Then, the smallest support among the remaining candidates is retained, choosing one at random if it is not unique.
\end{enumerate}

\begin{remark} 
We can view our problem as a linear regression such that the observation that we aim at regressing is null, $y=0$ and the design operator $\A\mapsto\widehat{\K}\A-\A\widehat{\K}$ is noisy. Our goal is to find a solution in the kernel of the operator $\A\mapsto{\K}\A-\A{\K}$. In this context, a Lasso procedure $(${\it i.e.}, minimizing $\|\widehat{\K}\A-\A\widehat{\K}\|^2+\lambda \| \A\|_1)$ without further constraints leads to the null matrix solution.  Therefore, we need to add a condition to avoid the null solution $\widehat\W=0$. Since we can only recover the target up to a scaling parameter, we should consider for instance that $\|\W\|=1$ and adding the constraint $\|\A\|=1$. It results in a non-convex program with no guarantees that a local minimum is the solution to the program. 

Recall that we aim at recovering the exact support when the number of observations is large. But using the $\ell_1$ penalty tends to overestimate the support and any conservative choice of $\lambda$ will lead to false positives in the estimated support. Furthermore, it can be understood that a full matrix may commute with $\widehat{\K}$, and at the same time it may have a small $\ell_1$ norm. That is to say that there may be no restricted eigenvalue condition for the noisy design operator in our framework. 

Hence, when aiming for support recovery, the typical solution is to vanish the small entries of $\widehat{\W}$, making it no more efficient than the thresholded $\ell_2$ procedure considered in Algorithm $2$. For this reason, the numerical performances of the Lasso procedure are not included in the study.
 \end{remark}

The next table compares the performances of the four algorithms. We calculated the Monte-Carlo estimated mean error $\mathbb E ( \delta(\widehat S)) $ and probability of exact recovery $\mathbb P\{\widehat S = S^\star\}$ for $1000$ repetitions of the experiment. The average computational time (obtained with the function \verb|timer| of Scilab) on a processor Intel Xeon $@2.6\,\text{GHz}$ are shown, using {the oracle} values of $\lambda$ and~$t$ for the first two algorithms (the calibration of these parameters is thus not accounted for in the computation time). 

\vspace{0.3cm}
\begin{tabular}{|c|c|c|c|c|}
\hline
 Algorithm 	& $\ell_0$ & $\ell_2 -$thresholding & Backward & Bagging Backward 
 \\ \hline Mean Error 		 & 0.45 	  &0.37       & 	1.95 & 0.68 	                                  
 \\ \hline Exact recoverery &    68\%      &  75\%        & 23\%          & 61\% 	
  \\ \hline CPU  time (s) &    0.32      &  0.002       & 0.009          &0.59 	\\ \hline
 \end{tabular}
\vspace{0.3cm}

In this example, the first two algorithms are the more accurate. The percentage of successful recoveries for the bagging backward algorithm is nonetheless competitive given that the first two procedures have been calibrated optimally for each experiment, which would be highly infeasible in practice. Finally, we observe that although it is much more expensive computationally, the bagging version of the backward algorithm yields an undeniable improvement.\\

Upper bounds for the time and space complexity of the algorithms are given in the next table. The time complexity is calculated as the number of different supports $S$ considered to lead to the solution in function of the size $N$ of the graph and the number $M$ of training samples. The spatial complexity measures the memory size needed to compute the solution. In this setting, it is the main limitation for applying the procedures to large graphs. The $N^4$ comes from the computation of $\Delta(\K) = \K\otimes \id - \id \otimes \K$ in the solver. Admittedly, the complexity could be improved by using sparse matrix encoding although this was not implemented.

 \vspace{0.3cm}
 \begin{tabular}{|c|c|c|c|c|}
 \hline
  Algorithm  & $\ell_0$            & $\ell_2-$thresholding & Backward & Bagging Backward   \\ \hline
 Space Complexity  &$O(N^4)$            & $O(N^4)$   & $O(N^4)$     &  $O(N^4)$    	  \\ \hline
Time Complexity&$O(2^{N(N-1)/2})$&   $O(1)$      &  $O(N^2)$          & $O(N^2.M)$                     \\ \hline
  \end{tabular}
	\vspace{0.3cm}

On the current version, the bagging backward algorithm contains scalability issues for big graphs due to its space complexity. Leads to reduce the spatial complexity include using sparse matrix encoding or the use of cheap approximations of the criterion. These shall be investigated in future works. 

\subsection{A diagonally identifiable matrix}\label{sec:simu2}

The advantages of the bagging backward algorithm are highlighted for larger graphs. In the next example, we consider the graph $G_2$ on $N=15$ vertices represented in Figure  \ref{fig:Graphetri}. The experimental conditions are similar to that of the previous example, a sample of size $n=10000$ is drawn from a centered Gaussian vector of variance $\K = \exp(\W)$ where $\W$ is the normalized adjacency matrix of $G_2$, with normalizing constant chosen such that $\mathbf 1^\top w = 1$. The implementation of the different algorithms follow the description of the previous example.

\begin{figure}[h!]
\center
\begin{tikzpicture}[line width =0.3mm,scale=0.60]
\vertex[circle,  minimum size=10pt](1) at  (0,0){} ;
\vertex[circle,  minimum size=10pt](2) at  (2,0) {};
\vertex[circle,  minimum size=10pt](3) at  (4,0){} ;
\vertex[circle,  minimum size=10pt](4) at  (6,0) {};
\vertex[circle,  minimum size=10pt](5) at (8,0){};
\vertex[circle,  minimum size=10pt](6) at (1,-1.5){};
\vertex[circle,  minimum size=10pt](7) at (3,-1.5){};
\vertex[circle,  minimum size=10pt](8) at (5,-1.5){};
\vertex[circle,  minimum size=10pt](9) at (7,-1.5){};
\vertex[circle,  minimum size=10pt](10) at (2,-3){};
\vertex[circle,  minimum size=10pt](11) at (4,-3){};
\vertex[circle,  minimum size=10pt](12) at (6,-3){};
\vertex[circle,  minimum size=10pt](13) at (3,-4.5){};
\vertex[circle,  minimum size=10pt](14) at (5,-4.5){};
\vertex[circle,  minimum size=10pt](15) at (4,-6){};

\draw (1)--(2)--(3)--(4)--(5)--(9)--(8)--(7)--(6)--(1) (9)--(12)--(14)--(15)--(13)--(10)--(11)--(12) ;
\draw (9)--(4) (6)--(10)  (12)--(8)--(3)--(11)--(7)--(2)  (13)--(14)--(11);
\end{tikzpicture}
\caption{The graph $G_2$ is diagonally identifiable.}
 \label{fig:Graphetri}
\end{figure}
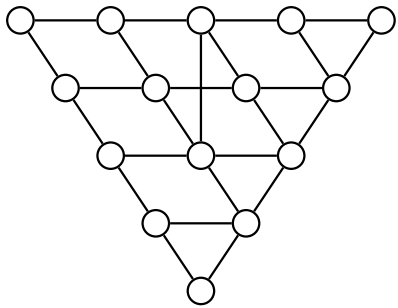

In this case, the number of possible supports is too large for the $\ell_0$ method to be implementable while the accuracy of the thresholded $\ell_2$ drops considerably compared to smaller cases. We summarize the results in the following table. 
 
 \vspace{0.3cm}
 {\begin{center}
 \begin{tabular}{|c|c|c|c|c|}
\hline
 Algorithm	 & $\ell_2 -$thresholding & Backward & Bagging Backward
 \\ \hline Mean Error 		  &      10  &  25         & 1	                                  
 \\ \hline Exact recoverery  &  22\%       &26\%        & 69\%	
  \\ \hline CPU  time   (s)          & 0.04      &  2.5        & 256	\\ \hline
 \end{tabular}
 \end{center}}
 \vspace{0.3cm}

A drawback of the bagging backward algorithm is the larger computational time: it takes around 4 minutes in average to estimate the support. Being essentially $M=100$ repetitions of the backward algorithm, the numerical complexity of the bagging version is roughly $M$ times that of the simple backward algorithm, although the improvement is, here again, clear.

 To illustrate the influence of the unknown function $f$, we consider $f: t \mapsto  (1-t)^{-2}$ and reproduce the numerical study for $\K=f(\W)$. The results for various sample sizes are gathered in the next table, for $M=100$ {bagging runs}. 

$$
\begin{array}{|c|c|c|c|c|c|c|c|}
\hline
n &10000&5000&2000&1000
\\ 
\hline
\text{Exact recovery} & 97\%  & 87\%   & 83\%  & 13\% \\
\hline 
\text{Mean error} &   0.05  & 0.33   & 0.9  & 8.5  \\
\hline 
\end{array} $$

The probability of recovering the true support appears to be greater than in the previous example ($97\%$ against $69\%$ previously for $n=10000$). This sheds lights on another important factor in the efficiency of the methods which is the separability of the spectrum of $\K$. Indeed, in this framework, the information needed to recover $\W$ lies in its eigenspaces, which are estimated via $\widehat \K$. The accuracy of these estimates depends on the distance between the different eigenvalues (see e.g. Corollary~4.12 in \cite{stewart1990matrix} and the Wedin'\! $\sin(\theta)$ theorem in \cite{stewart1990matrix}). Thus, for $\lambda_1,...,\lambda_N$ the spectrum of $\W$, the ability to recover $\W$ from $\K=f(\W)$ is strongly impacted by the distances $ | f(\lambda_i) - f(\lambda_j)|, i,j=1,...,N$. In this situation where $f$ is the exponential function, $\W$ having few negative eigenvalues will thus have a positive impact on the estimation. For the sake of comparison, the spectrum of $\W$, given by $\{- 0.45,- 0.28 ,- 0.26 ,- 0.21, - 0.18 ,- 0.16 ,- 0.08 ,- 0.01 , 0.03,0.10  , 0.14  , 0.19,   0.31 ,  0.34,   0.52\}$, is more ``spread'' by the function $t \mapsto (1-t)^{-2}$ than by the exponential, see Figure \ref{fig:f}.

\begin{figure}[H]
	\centering
		\includegraphics[width=0.72\textwidth]{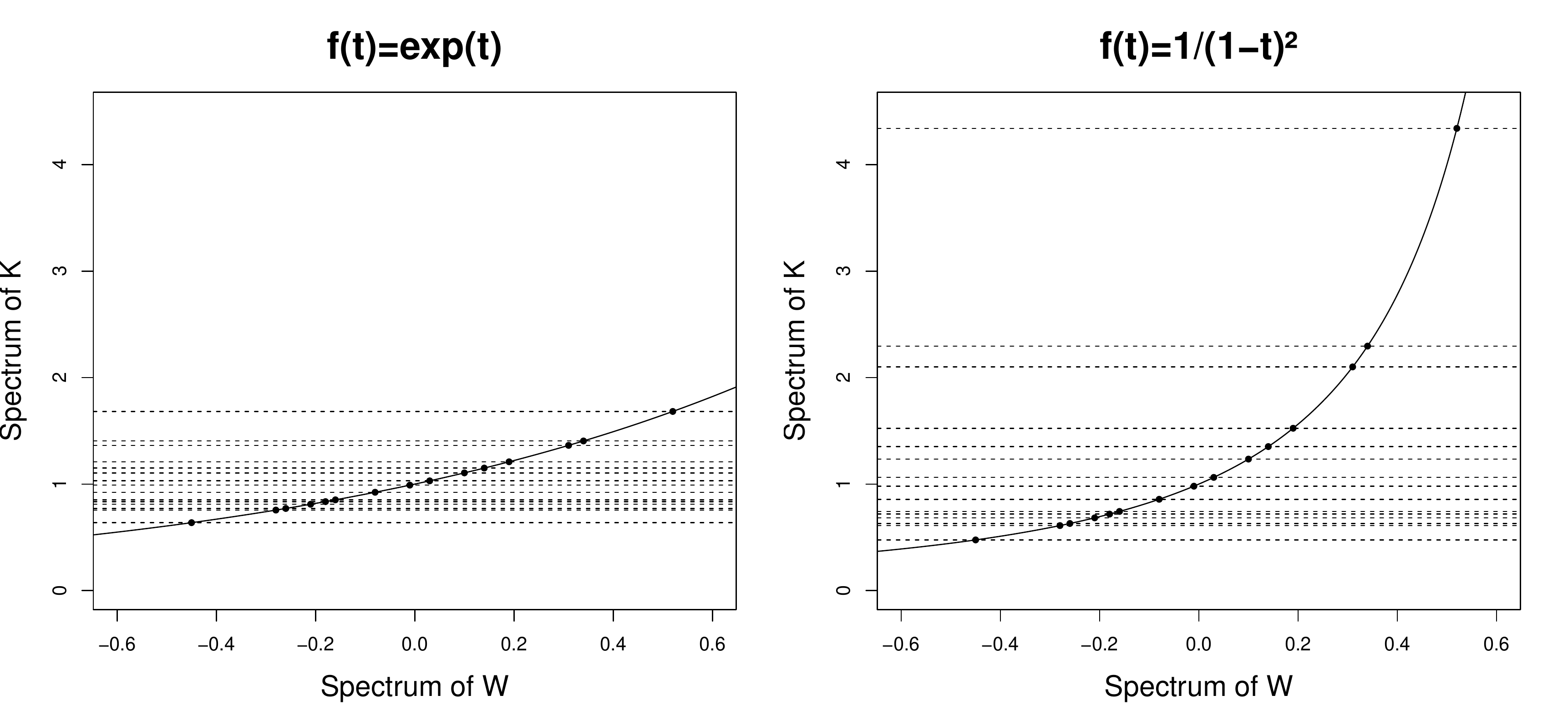}
		\caption{Separability of the spectrum of $\K=f(\W)$ for $f: t \mapsto \exp(t)$ (left) and $f: t \mapsto (1-t)^{-2}$ (right). The eigenvalues of $\K$ are more separated in the second case, making it easier to approximate its eigenspaces from the estimator $\widehat \K$.}
	\label{fig:f}
\end{figure}

\begin{remark} We also implemented the procedure in a random setting where $\W$ is drawn from an Erd\"os-R\'enyi graph with binomial entries. The conclusions obtained in this case are similar to those already discussed and shall not be presented to avoid redundancy.
\end{remark}

\section{Real life application}
\label{sec:reallife}
We now implement the bagging backward algorithm on real life data provided by M\'et\'eorage and M\'et\'eo France. The data contain the daily number of lightnings during a $3$ year period in~$16$ regions of France localized on a $4\times 4$ grid. We expect to recover the spatial structure of the graph from the dependence of the lightning occurrences between the regions.


The data are refined as follows. We first eliminate day without any lighting all over France and we obtain some observations $X_{i}$, $i = 1, \ldots, 950$, where $X_i$ is a vector of length $16$ giving the number of impacts at day $i$ in each of the $16$ regions. This numbers are highly non Gaussian, contain many zeros, and show a clear south-east/north-west tendency (with much more lightning in the south east). Therefore, we look at the numbers at the log scale (taking $\log(X_i+1)$, with~$+1$ dealing with vanishing values $X_i=0$) and we subtracted the spatial tendency (this operation replaces vanishing values by small residues after regression). Now, it remains a strong inhomogeneity, that should violate the assumption that the underlying graph has no self-loops ({\it i.e.},~the diagonal of $\W$ is zero). To overcome this problem, we normalize the process in such manner that the conditional variance at each vertex conditionally to all the other is $1$. 

We model the resulting process as a spatial AutoRegressive Gaussian fields, as described in Section \ref{sec:SpatialAR}. Given that the covariance matrix of the process commutes with the underlying graph, we applied our algorithm: we draw $100$ learning samples, keeping or not each observation with probability $1/2$, and retained $20\%$ of the $100$ trajectories. We do not obtain exactly the same graph at each run, although the graphs are most of the time very satisfying. To show the results, we ran $100$ times the algorithm and kept, for every edge, the proportion of the time that this edge appears. This is summarized in Figure \ref{fig:prop}.

\begin{figure}[H]
\center
\begin{tikzpicture}[line width =0.3mm,scale=0.5]
\vertex[circle,  minimum size=7pt](11) at  (0,0){} ;
\vertex[circle,  minimum size=7pt](12) at  (2,0) {};
\vertex[circle,  minimum size=7pt](13) at  (4,0){} ;
\vertex[circle,  minimum size=7pt](14) at  (6,0) {};
\vertex[circle,  minimum size=7pt](21) at  (0,-2){} ;
\vertex[circle,  minimum size=7pt](22) at  (2,-2) {};
\vertex[circle,  minimum size=7pt](23) at  (4,-2){} ;
\vertex[circle,  minimum size=7pt](24) at  (6,-2) {};
\vertex[circle,  minimum size=7pt](31) at  (0,-4){} ;
\vertex[circle,  minimum size=7pt](32) at  (2,-4) {};
\vertex[circle,  minimum size=7pt](33) at  (4,-4){} ;
\vertex[circle,  minimum size=7pt](34) at  (6,-4) {};
\vertex[circle,  minimum size=7pt](41) at  (0,-6){} ;
\vertex[circle,  minimum size=7pt](42) at  (2,-6) {};
\vertex[circle,  minimum size=7pt](43) at  (4,-6){} ;
\vertex[circle,  minimum size=7pt](44) at  (6,-6) {};

\draw (11)--(12)--(13)--(14)--(24) (23)--(22)--(21)--(11) (22)--(12) (31)--(21) (32)--(33)--(34)--(24) (33)--(23)--(22)--(32)--(31)--(22)
(31)--(41)--(42)--(43) (44)--(34) (42)--(32) (44)--(43)--(33)  (41) (14) (43) (23)--(13) (22)--(13) (31)--(12)  (33)--(24)--(41)--(22) (41)--(23) (22)--(13) 
(13)--(22) (23)--(24);

\begin{scope}[xshift=10cm]
\vertex[circle,  minimum size=7pt](11) at  (0,0){} ;
\vertex[circle,  minimum size=7pt](12) at  (2,0) {};
\vertex[circle,  minimum size=7pt](13) at  (4,0){} ;
\vertex[circle,  minimum size=7pt](14) at  (6,0) {};
\vertex[circle,  minimum size=7pt](21) at  (0,-2){} ;
\vertex[circle,  minimum size=7pt](22) at  (2,-2) {};
\vertex[circle,  minimum size=7pt](23) at  (4,-2){} ;
\vertex[circle,  minimum size=7pt](24) at  (6,-2) {};
\vertex[circle,  minimum size=7pt](31) at  (0,-4){} ;
\vertex[circle,  minimum size=7pt](32) at  (2,-4) {};
\vertex[circle,  minimum size=7pt](33) at  (4,-4){} ;
\vertex[circle,  minimum size=7pt](34) at  (6,-4) {};
\vertex[circle,  minimum size=7pt](41) at  (0,-6){} ;
\vertex[circle,  minimum size=7pt](42) at  (2,-6) {};
\vertex[circle,  minimum size=7pt](43) at  (4,-6){} ;
\vertex[circle,  minimum size=7pt](44) at  (6,-6) {};

\draw (11)--(12)--(13)--(14)--(24) (23)--(22)--(21)--(11) (22)--(12) (31)--(21) (32)--(33)--(34)--(24) (33)--(23)--(22)--(32)--(31) (22)--(13)
(31)--(41)--(42)--(43) (44)--(34) (42)--(32) (43)--(33)  (41) (14) (43) (23)--(13) ;
\end{scope}

\begin{scope}[xshift=20cm]
\vertex[circle,  minimum size=7pt](11) at  (0,0){} ;
\vertex[circle,  minimum size=7pt](12) at  (2,0) {};
\vertex[circle,  minimum size=7pt](13) at  (4,0){} ;
\vertex[circle,  minimum size=7pt](14) at  (6,0) {};
\vertex[circle,  minimum size=7pt](21) at  (0,-2){} ;
\vertex[circle,  minimum size=7pt](22) at  (2,-2) {};
\vertex[circle,  minimum size=7pt](23) at  (4,-2){} ;
\vertex[circle,  minimum size=7pt](24) at  (6,-2) {};
\vertex[circle,  minimum size=7pt](31) at  (0,-4){} ;
\vertex[circle,  minimum size=7pt](32) at  (2,-4) {};
\vertex[circle,  minimum size=7pt](33) at  (4,-4){} ;
\vertex[circle,  minimum size=7pt](34) at  (6,-4) {};
\vertex[circle,  minimum size=7pt](41) at  (0,-6){} ;
\vertex[circle,  minimum size=7pt](42) at  (2,-6) {};
\vertex[circle,  minimum size=7pt](43) at  (4,-6){} ;
\vertex[circle,  minimum size=7pt](44) at  (6,-6) {};

\draw (11)--(21)--(31)--(41) (42)--(32)--(22)--(12)--(13)--(23)--(33)--(43) (13)--(14)--(24)--(34)--(44);
\end{scope}
\end{tikzpicture}
\caption{Edges that appears in $30\%$, $50\%$, and $70\%$ of the time in the bagging backward algorithm.}
 \label{fig:prop}
\end{figure}
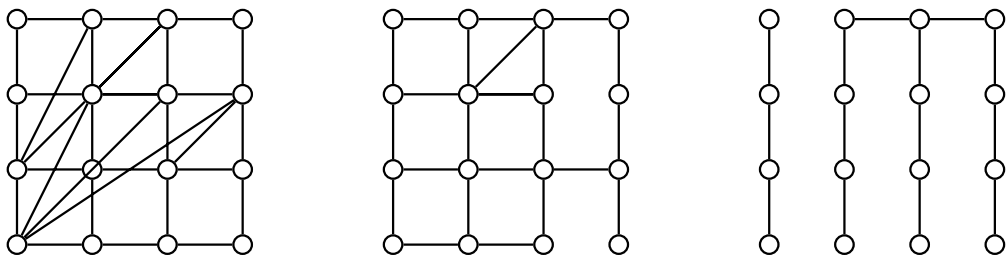

To compare the performance of our method, we used the package \texttt{GGMselect} to infer Graphical Models, see \cite{giraud2012graph}. This package is very efficient, and powerful even for samples with more vertices than observations. It is not designed exactly for our case, so we do not pretend that our method makes better than this algorithm. Furthermore, we did not tune the parameters, and used rather the default parameters, only specifying the maximal degree of each vertex as $\texttt{dmax} = 5$ and the family \texttt{CO1}. The results are given in Figure~\ref{fig:resultggm}.

We also implemented \texttt{GGMselect} on learning sample obtained keeping observation with probability $1/2$ (as for the bagging backward algorihm), and represent how often an edge appeared, as in our method. We have to note that \texttt{GGMselect} seems more robust than our method, and this fact holds also for the other family \texttt{LA} even if we will not present here the quite similar results. Furthermore, our algorithm takes a lot of time compared to \texttt{GGMselect} ($0.3 s.$ for \texttt{GGMselect} and $400 s.$ for our algorithm.)

Nevertheless, the two methods give different results. The normalized lightning fields happens to be closed to a Simultaneous AutoRegressive process of order $k$ on $\mathbb Z^2$ (see for instance \cite{guyon1995random} and \cite{gaetan2010spatial}). Hence, we expect that the target graph to be alike $\mathbb Z^2$. In Figure~\ref{fig:prop}, we observe that the $50\%$ present edges graph seems to uncover this spatial dependency keeping only edges between adjacent regions. Note that, the package \texttt{GGMselect} aims to recover a weighted graph of the paths of length at most $k$ on the grid while our method aims to recover the graph itself.

\begin{figure}[h]
 \centering

\begin{tikzpicture}[line width =0.3mm,scale=0.5]
\vertex[circle,  minimum size=7pt](11) at  (0,0){} ;
\vertex[circle,  minimum size=7pt](12) at  (2,0) {};
\vertex[circle,  minimum size=7pt](13) at  (4,0){} ;
\vertex[circle,  minimum size=7pt](14) at  (6,0) {};
\vertex[circle,  minimum size=7pt](21) at  (0,-2){} ;
\vertex[circle,  minimum size=7pt](22) at  (2,-2) {};
\vertex[circle,  minimum size=7pt](23) at  (4,-2){} ;
\vertex[circle,  minimum size=7pt](24) at  (6,-2) {};
\vertex[circle,  minimum size=7pt](31) at  (0,-4){} ;
\vertex[circle,  minimum size=7pt](32) at  (2,-4) {};
\vertex[circle,  minimum size=7pt](33) at  (4,-4){} ;
\vertex[circle,  minimum size=7pt](34) at  (6,-4) {};
\vertex[circle,  minimum size=7pt](41) at  (0,-6){} ;
\vertex[circle,  minimum size=7pt](42) at  (2,-6) {};
\vertex[circle,  minimum size=7pt](43) at  (4,-6){} ;
\vertex[circle,  minimum size=7pt](44) at  (6,-6) {};

\draw (11)--(13)--(23)--(43)--(11)--(31)--(41)--(42)--(24)--(14)--(13)--(21)--(12) (11)--(44)--(32)--(24)--(42)--(34)--(32) (12)--(41)--(13)
(31)--(14) (23)--(41)--(24) (43)--(24) (23)--(44);

\begin{scope}[xshift=10cm]
\vertex[circle,  minimum size=7pt](11) at  (0,0){} ;
\vertex[circle,  minimum size=7pt](12) at  (2,0) {};
\vertex[circle,  minimum size=7pt](13) at  (4,0){} ;
\vertex[circle,  minimum size=7pt](14) at  (6,0) {};
\vertex[circle,  minimum size=7pt](21) at  (0,-2){} ;
\vertex[circle,  minimum size=7pt](22) at  (2,-2) {};
\vertex[circle,  minimum size=7pt](23) at  (4,-2){} ;
\vertex[circle,  minimum size=7pt](24) at  (6,-2) {};
\vertex[circle,  minimum size=7pt](31) at  (0,-4){} ;
\vertex[circle,  minimum size=7pt](32) at  (2,-4) {};
\vertex[circle,  minimum size=7pt](33) at  (4,-4){} ;
\vertex[circle,  minimum size=7pt](34) at  (6,-4) {};
\vertex[circle,  minimum size=7pt](41) at  (0,-6){} ;
\vertex[circle,  minimum size=7pt](42) at  (2,-6) {};
\vertex[circle,  minimum size=7pt](43) at  (4,-6){} ;
\vertex[circle,  minimum size=7pt](44) at  (6,-6) {};

\draw (11)--(13)--(23)--(43)--(11)--(31)--(41)--(42)--(24)--(14)--(13)--(21)--(12) (11)--(44)--(32)--(24)--(42)--(34)--(32) (12)--(41)--(13);
\end{scope}

\begin{scope}[xshift=20cm]
\vertex[circle,  minimum size=7pt](11) at  (0,0){} ;
\vertex[circle,  minimum size=7pt](12) at  (2,0) {};
\vertex[circle,  minimum size=7pt](13) at  (4,0){} ;
\vertex[circle,  minimum size=7pt](14) at  (6,0) {};
\vertex[circle,  minimum size=7pt](21) at  (0,-2){} ;
\vertex[circle,  minimum size=7pt](22) at  (2,-2) {};
\vertex[circle,  minimum size=7pt](23) at  (4,-2){} ;
\vertex[circle,  minimum size=7pt](24) at  (6,-2) {};
\vertex[circle,  minimum size=7pt](31) at  (0,-4){} ;
\vertex[circle,  minimum size=7pt](32) at  (2,-4) {};
\vertex[circle,  minimum size=7pt](33) at  (4,-4){} ;
\vertex[circle,  minimum size=7pt](34) at  (6,-4) {};
\vertex[circle,  minimum size=7pt](41) at  (0,-6){} ;
\vertex[circle,  minimum size=7pt](42) at  (2,-6) {};
\vertex[circle,  minimum size=7pt](43) at  (4,-6){} ;
\vertex[circle,  minimum size=7pt](44) at  (6,-6) {};

\draw (11)--(13)--(23)--(43)--(11)--(31)--(41)--(42)--(24)--(14)--(13)--(21)--(12) (11)--(44)--(32)--(24)--(42)--(34)--(32);
\end{scope}
\end{tikzpicture}

 \caption{Results with the \texttt{GGMselect} package, with families \texttt{LA}, \texttt{CO1} and \texttt{QE}.}
 \label{fig:resultggm}
\end{figure}
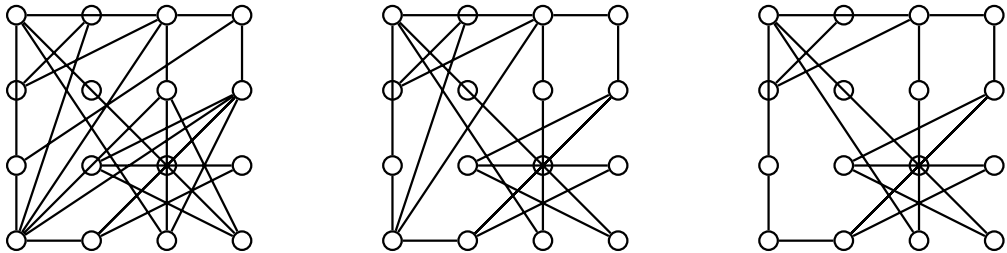

The results show that, in this case, and with the purpose of finding an underlying graph that ``generates'' the process, our method seems to work at least as fine as an inference of a graphical model, modeling data as a Gaussian Markov Field. We insist that we do not claim this fact to be general. In particular, we need much more observations than the methods developed in this package. But we pretend that, in different contexts, and with enough observations, we can be as good as other methods. Indeed, our method yet presents one advantage: the process does not need to be Markov, and for instance, we could infer spatial autoregressive process of any order (whereas graphical model inference can only recover underlying graphs for $\mathrm{AR}_1$ spatial processes, which are Markov). But this advantage turns into a problem when the process is truly Markov, because we do not use the knowledge of the function $f$, which can be taken as $1/x$ in the Markov case.

\section{Discussion}
\label{sec:discussion}

In this paper, we develop a new method to recover hidden graphical structures in different models that shares the fact that, one way or another, we have access to an approximation of the eigenstructure of the graph, through an estimation of an operator that commutes with a weighted adjacency matrix of this unknown graph. This is noticeable that we do not need any sparsity assumption to make the method work, and even with the large number of unknown parameters ($\K = f(\W)$, with the support, the function $f$, and the non-null entries of $\W$ are all unknown), we can perfectly recover the support when enough observations are available. We only assume that we know the location of some zeros. The most interesting case is when the known zeros are localized onto the diagonal, because it only means that the process is well normalized, in a sense, because all self-loops have same weights.

Note that there is a number of observations below which the algorithm always provides a wrong support. Furthermore, this fact can be observed in practice, because almost all learning samples will lead to different supports. This limit is intrinsic to our model and is a matter of balance between the sample noise $\Vert \widehat{\K} - \K \Vert$ and the signal strength. The noise is the estimation error of $\K$, and has order $1/{\sqrt{n}}$, whereas the signal is of order $c_0(S^\star)$, see \eqref{eq:def_c0} and~\eqref{eq:threshold}. 

Furthermore, the paper addresses the problem of exact support recovery, which is way harder than to provide an approximation of the support. The performances presented in this paper were computed with defaults parameters, but manual tuning seems to improve a little bit the results. In particular, drawing learning samples with probability $\frac12$ may cause overfitting, and for very large samples, we do not always get $100\%$ exact support recovery. This problem can be easily bypassed by either decreasing the size of learning samples, or increasing the thresholds. In the present version, $3$ parameters have been empirically chosen : the size of learning samples, the number of bagging trajectories, and the way we regroup the results of theses bagging trajectories. One challenge for future work is to justify theses choices with theoretical results.

For practical issues, there remain three other challenges that have to be bypassed. The first one concern the assumption about the symmetry of $\W$, that should be released for real practical interest. The second concerns the assumption that $\W$ has a null diagonal. It remains to find an effective way to normalize the process when this assumption does not hold (the normalization used in Section \ref{sec:reallife} assume an autoregressive structure). Finally, our algorithm is greedy when the size of the graphs increases, and for large graphs, it would be really interesting to find a way to compute a cheap version of the criterion, and to compute the significance of the variable.


\section*{Acknowledgement}

The authors would like to thank Dieter Mitsche for fruitful discussions. We would like to warmly thank M\'et\'eo France et M\'et\'eorage for providing us the data used in Section~\ref{sec:reallife}. We would like to thank the Universidad de la Habana (Cuba) and the Centro de Modelamiento Matematico (Chile) for their hospitality.


\appendix

\section{Asserting the Diagonal Identifiability}\label{s:disc}

\subsection{Necessary and sufficient conditions}
\label{sec:CNS}
In this section, we focus on the $F$-identifiability in the special case where the set of forbidden entries is the diagonal $F_{\operatorname{diag}}:=\{ (i,i): i \in [1,N] \}$. Recall that a support $S$ is $F_{\operatorname{diag}}$-identifiable, or simply diagonally identifiable (DI), if for almost every matrix $\mathsf{A} \in \mathcal E(S)$, 
$$ \mathsf{B} \mathsf{A} = \mathsf{A} \mathsf{B} \ , \ \operatorname{diag}(\mathsf{B}) = 0 \ , \ \mathsf B = \mathsf B^\top \  \Longrightarrow \ \mathsf{B} = \lambda\mathsf{A}    $$
for some $\lambda \in \mathbb R$. In other words, a support $S$ is diagonally identifiable if almost every symmetric matrix $\mathsf{A}$ with support in $S$ is uniquely determined, up to scaling, by its eigenspaces among symmetric matrices with zero diagonal. In this section, we provide both sufficient and necessary conditions on a support $S$ to ensure the $F_{\operatorname{diag}}$-identifiability. For this, we consider a simple undirected graph $G_S = ([1,N],S)$ on $N$ vertices with edge set $S$.
\begin{definition}[{Induced subgraph}] For $V \subseteq [1,N]$, the induced subgraph $G_S(V) = (V,S(V))$ is the graph on $V$ with edge set $S(V) = S \cap V^2$. 
\end{definition}
\begin{proposition}
\label{p:invert}
For all support $S\subseteq [1,N]^2$, the set of invertible matrices in ${\mathcal E(S)}$ is either empty or a dense open subset of $\mathcal E(S)$.
\end{proposition}
The proof is straightforward when writing the determinant of $\mathsf{A} \in \mathcal E(S)$ as a polynomial in its entries. Observe that by this property, finding one invertible matrix $\mathsf{A}$ in $\mathcal E(S)$ guarantees that almost every matrix in $\mathcal E(S)$ is invertible. In this case, we say that the graph $G_S$ is invertible. Similarly, we say that $G_S$ is diagonally identifiable if $S$ is diagonally identifiable.

\begin{theorem}[Conditions for $F_{\operatorname{diag}}$-identifiability]
\label{t:iden}
Let $S\subseteq\overline F_{\operatorname{diag}}$\, and $G_S=([1,N],S)$. 
\begin{enumerate}	
	\item \textbf{Necessary condition:} If $S$ is diagonally identifiable then there exists a sequence of subsets $V_3,...,V_{N-1} \subset [1,N]$ such that $| V_k | = k$ and $G_S(V_k)$ is invertible for all $k=3,...,N-1$.
	\item \textbf{Sufficient condition:} If there exists a nested sequence $V_3 \subset ... \subset V_{N-1} \subset [1,N]$ with $\vert V_k \vert =k$ such that	$G_S(V_k)$ is invertible for all $k=3,...,N-1$, then $S$ is diagonally identifiable.
\end{enumerate}
\end{theorem}
The gap between the sufficient and necessary conditions lies essentially in the fact that the sequence $V_3,...,V_{N-1}$ need to be nested for the sufficient condition. 

\bigskip 
\begin{proof} We proceed by contradiction. For the necessary condition, let $k\geq 3$ be such that $G_S(V_k)$ is not invertible, for all subset $V_k \subset [1,N]$ of size $k$. For $\mathsf{A} \in \mathcal E(S)$, denote by $\psi_0(\mathsf{A}),\psi_1(\mathsf{A}),\ldots,\psi_N(\mathsf{A})$ the coefficients of the characteristic polynomial
$$ \det(z \id - \mathsf{A}) =  \sum_{j=0}^N \psi_j(\mathsf{A}) \ z^j \!, \quad z \in \mathbb R.  $$
Consider the matrix $M_k(\mathsf{A}) :=\sum_{j=0}^{k} \psi_j(\mathsf{A}) \ \mathsf A^j$. By Eq.\! (14) in \cite{espinasse2015relations}, we see that the $(i,i)$-entry of $M_k(\mathsf{A})$ equals the sum of all minors of size $k$ that do not contain the vertex $i$. Thus, the condition that $G_S(V_k)$ is not invertible for all subset $V_k$ of size $k$ implies that $M_k(\mathsf{A})$ has zero diagonal. On the other hand, the non-zero entries of $M_k(\mathsf A)$ are degree~$k$ polynomials in the variables $\mathsf A_{ij}, (i,j) \in \supp(\mathsf A)$. Therefore, the equality $M_k(\mathsf{A}) = \lambda \mathsf{A}$ for some $\lambda \in \mathbb R$ occurs for at most a countable number of $\mathsf{A} \in \mathcal E(S)$. Since $M_k(\mathsf{A})$ commutes with $\mathsf{A}$, we deduce that $S$ is not diagonally identifiable.\\

\noindent For the sufficient condition, we will need the following lemma.
\begin{lemma}\label{l:ident}
 If there exists a subset $V' \subset [1,N]$ of size $N-1$ such that $G_S(V')$ is both DI and invertible, then $G_S$ is DI.
\end{lemma}

\begin{proof} We may assume that $V' = [1,N-1]$ without loss of generality. Let $\mathsf{M}'$ denote a symmetric $(N-1)\times (N-1)$ matrix indexed on $V'$ that is both invertible and diagonally identifiable, \textit{i.e.},~for all non-zero matrix  $\mathsf{A}' \neq \lambda \mathsf M'$,
$$ \mathsf{M}' \mathsf{A}' = \mathsf{A}' \mathsf{M}'  \ \Longrightarrow \ \operatorname{diag}(\mathsf{A}') \neq 0 . $$
To prove that $G_S$ is DI, it suffices to find a symmetric matrix $\mathsf M$ with support $S$ that is diagonally identifiable. Consider $\mathsf{M}$ defined by
\begin{align*}
 \mathsf{M} = \left[\begin{matrix}
     \mathsf{M}' & 0 \\ 0 & 0
    \end{matrix}\right] .
\end{align*}
Let $\mathsf{A}$ be a matrix with zero diagonal that commutes with $\mathsf M$ and write
\begin{align*}
\mathsf{A} = \left[\begin{matrix}  \mathsf{A}' & a \\ a^\top & 0
    \end{matrix}\right]
\end{align*}
for some $a \in \mathbb R^{N-1}$, with diag$(\mathsf{A}')=0$. The condition $\mathsf{M}\mathsf{A}=\mathsf{A}\mathsf{M}$ can be stated equivalently as
$$ \left\{ \begin{array}{l} \mathsf{M}'\mathsf{A}'=\mathsf{A}'\mathsf{M}' \\ \mathsf{M}' a = 0  \end{array} \right.  $$
Since $\mathsf{M}'$ is invertible by assumption, $a=0$ and the only matrix $\mathsf{A}$ with zero diagonal that commutes with $\mathsf{M}$ is the null matrix. Thus, $\mathsf{M}$ is diagonally identifiable.
\end{proof}

We now go back to prove the sufficient condition in Theorem \ref{t:iden}. Assume that $G_S$ is not diagonally identifiable, then by Lemma \ref{l:ident}, neither is $G_S(V_{N-1})$. By iterating the argument, we conclude that $G_S(V_3)$ is not diagonally identifiable. However, the only invertible graph on three vertices is the triangle graph, which is diagonally identifiable, leading to a contradiction. 
\end{proof}

\begin{remark}
The proof of Theorem~\ref{t:iden} {combines the results of} Lemma 2.1 in \cite{fytp14} and  Eq.\! (14) in \cite{espinasse2015relations}. The first one is of topological flavor proving that the set of identifiable matrices is either dense or empty in the set  of matrices with prescribed support. The paper \cite{fytp14} does not address condition on identifiability and Lemma 2.1 in \cite{fytp14} is not an identifiability result. The second ingredient is Eq.\! (14) in \cite{espinasse2015relations}. Actually, the paper  \cite{espinasse2015relations} contains a key combinatorial computation on the adjugate matrix of weighted graphs and, we must confess, it has been motivated by addressing a combinatorial calculus in the proof of identifiability. It gives part of the present proof (it proves that $M_k(\mathsf A)$ has zero diagonal in the proof of the necessary condition) but it is far from being its essence. The proof of the sufficient condition does not involve this calculus and proving the necessary part requires other simple but non trivial steps.
\end{remark}

\subsection{Proof of Proposition~\ref{prop:KiteGeneral}}
\label{sec:kites_suffice}
From Claim $(ii)$ in Theorem~\ref{t:iden} and considering the nested sequence $V_{N-1}\supset ... \supset V_3$ obtained by removing the last vertex on the tail of the kite at each step, we deduce a simple and tractable sufficient condition for a graph~$G_S$ to be diagonally identifiable, namely that~$G_S$ contains the kite graph as a vertex covering (possibly not induced) subgraph.

\subsection{Existence of kites}
\label{s:kite}
The condition on containing the kite graph $\nabla_N$ as a subgraph is mild in the sense that it is satisfied in the dense regime $\log n/n$ by random graphs, as depicted in the following proposition.
\begin{proposition}
\label{prop:kite}
The existence of kite graphs in the Erd\H{o}s-R\'enyi model occurs as follows. For any $\omega(N)\to\infty$ and for $G_N\sim G(N,p_N)$, if $p_N\geq(1/N)({\log N}+{\log\log N}+\omega(N))$ then $\mathbb P\{G_N\ \mathrm{has\ a\ kite\ of\ length\ }N\}$ tends to $1$ as $N$ goes to infinity.
\end{proposition}
The proof makes use of the existence of a hamiltonian cycle which is a standard result in Random Graph Theory, see Corollary 8.12 in \cite{bollobas1998random} for instance. This results shows that in the regime $({\log N}+{\log\log N})/N$ an Erd\H{o}s-R\'enyi graph is diagonally identifiable. 

\medskip 

\begin{proof}
We now present the proof of this fact. Let $\omega(n)\to\infty$ and set
\begin{align*}
p_1&:=(1/n)({\log n}+{\log\log n}+{\omega(n)}/2),\\
p_2&:={\omega(n)}/(2n)\,.
\end{align*}
Let $G^{(1)}$ and $G^{(2)}$ be two independent Erd\H{o}s-R\'enyi graphs such that
\[
G_n^{(1)}\sim G(n,p_1)\quad
\independent
\quad G_n^{(2)}\sim G(n,p_2)\,.
\]
As shown in Corollary 8.12 in \cite{bollobas1998random} for instance, $\mathbb P\{G^{(1)}_n\ \mathrm{is\ hamiltonian}\}$ tends to $1$ as~$n$ goes to infinity. Given a hamiltonian cycle $C_n$ of length $n$ in $G^{(1)}$ one can construct a kite of length $n$ using edges of $G^{(2)}$ to connect a pair of vertices at distance $2$ on the cycle~$C_n$. Invoke the independence of $G^{(1)}$ and $G^{(2)}$ to get that this latter probability is 
\[
\mathbb P\{\{k,k+2\}\ \mathrm{is\ an\ edge\ of}\ G^{(2)}\ \mathrm{for\ some\ }k\}=\mathbb P\{B(n,p_2)>0\}\,,
\]
 where $B(n,p_2)$ denotes the binomial law. Using Poisson approximation one gets that this probability tends to $1$ as $n$ goes to infinity. We deduce that the probability that the graph $G=G_n^{(1)}+G_n^{(2)}$ has at least a kite tends to $1$. Observe that $G$ is an Erd\H{o}s-R\'enyi graph of size~$n$ and parameter $p=p_1+p_2-p_1p_2\leq p_n$ which concludes the proof.
\end{proof}

 \subsection{Proof of Theorem \ref{t:invertER}}
 \label{sec:invertER}
Combining Proposition \ref{prop:kite} and Theorem \ref{t:iden}, we deduce the first point. In view of the first point of Theorem \ref{t:iden}, we see that it is sufficient to find two isolated vertices to prove non-identifiability. Indeed, in this case, the kernel of the adjacency matrix has co-dimension at least $2$ showing that all sub-graphs of size $N-1$ are not invertible. Furthermore, one knows (see Theorem 3.1 in \cite{bollobas1998random} for instance) that the event ``there is at least two isolated points'' has sharp threshold function $\log n/n$. It proves the second point.

\section{Support reconstruction}
\subsection{Proof of Theorem \ref{t:main1}}
\label{proof:ThmL0}
Define $\mathcal S_1 := \{ S \in \mathcal S: \vert S \vert \leq \vert S^\star  \vert, S \neq S^\star  \} $ and $\mathcal S_2 := \{ S \in \mathcal S: \vert S \vert > \vert S^\star  \vert \} $, clearly it holds $\mathcal S = \{ S^\star  \} \cup \mathcal S_1 \cup \mathcal S_2$. We want to control the terms $\mathbb P\{\widehat S \in \mathcal S_1 \}$ and $\mathbb P\{\widehat S \in \mathcal S_2 \}$ separately and conclude in view of
\[
\mathbb P\{\widehat S \neq S^\star  \}= \mathbb P\{\widehat S \in \mathcal S_1 \} + \mathbb P\{\widehat S \in \mathcal S_2 \}\,.
\]
Since the Frobenius norm is sub-multiplicative, it holds, for all $\mathsf{A}\in\mathcal E(\overline F)$, 
\[
\Vert \mathsf{A} (\widehat{\mathsf{K}} - \mathsf{K}) - (\widehat{\mathsf{K}} - \mathsf{K}) \mathsf{A} \Vert \leq \Vert \mathsf{A} (\widehat{\mathsf{K}} - \mathsf{K}) \Vert_{2} + \Vert (\widehat{\mathsf{K}} - \mathsf{K}) \mathsf{A} \Vert \leq 2 \Vert \mathsf{A} \Vert \Vert \widehat{\mathsf{K}} - \mathsf{K} \Vert\,.
\]
Thus, the quantity $ \Vert \mathsf{A} \widehat{\mathsf{K}}- \widehat{\mathsf{K}}\mathsf{A} \Vert $ for $\mathsf{A} \in \mathcal E(\overline F)$ can be bounded from below and above by
\begin{equation}\label{cs} \Vert \mathsf{A}\mathsf{K}-\mathsf{K}\mathsf{A}  \Vert - 2\Vert \mathsf{A} \Vert \Vert \widehat{\mathsf{K}}- \mathsf{K}  \Vert \leq \Vert \mathsf{A} \widehat{\mathsf{K}}- \widehat{\mathsf{K}}\mathsf{A}  \Vert \leq \Vert \mathsf{A}\mathsf{K}-\mathsf{K}\mathsf{A}  \Vert + 2 \Vert \mathsf{A} \Vert \Vert \widehat{\mathsf{K}}- \mathsf{K}  \Vert. 
\end{equation}
To bound the term $\mathbb P\{\widehat S \in \mathcal S_1 \}$, we use \eqref{cs} to remark that for all $S \in \mathcal S_1$,
\[
 Q(S) = \min_{\mathsf{A} \in \mathcal E(S)\setminus\{0\}} \frac{\Vert \mathsf{A} \widehat{\mathsf{K}}- \widehat{\mathsf{K}}\mathsf{A}  \Vert}{\Vert \mathsf{A} \Vert} + \lambda_n \vert S \vert \geq  \min_{\mathsf{A} \in \mathcal E(S)\setminus\{0\}} \frac{\Vert \mathsf{A}\mathsf{K}-\mathsf{K}\mathsf{A}  \Vert}{\Vert \mathsf{A} \Vert} - 2 \Vert \widehat{\mathsf{K}}-\mathsf{K}  \Vert\,.
 \]
It follows
\begin{equation}\label{mins1} \min_{S \in \mathcal S_1} Q(S) \geq \min_{S \in \mathcal S_1} \min_{\mathsf{A} \in \mathcal E(S)\setminus\{0\}} \frac{\Vert \mathsf{A}\mathsf{K}-\mathsf{K}\mathsf{A}  \Vert}{\Vert \mathsf{A} \Vert} - 2 \Vert \widehat{\mathsf{K}}-\mathsf{K}  \Vert = c_0(S^\star) - 2 \Vert \widehat{\mathsf{K}}-\mathsf{K}  \Vert. \end{equation}
The constant $c_0(S^\star)$ is positive by $F$-identifiability of $\W$. Moreover, observe that
\begin{equation}
\label{min*} 
Q(S^\star ) = \min_{\mathsf{A} \in \mathcal E({S^\star })\setminus\{0\}} \frac{\Vert \mathsf{A} \widehat{\mathsf{K}}- \widehat{\mathsf{K}}\mathsf{A}  \Vert}{\Vert \mathsf{A} \Vert} + \lambda_n \vert S^\star  \vert \leq \frac{\Vert \W \widehat{\mathsf{K}}- \widehat{\mathsf{K}}\W  \Vert}{\Vert \mathsf{W} \Vert} + \lambda_n \vert S^\star  \vert \leq 2 \Vert \widehat{\mathsf{K}}- \mathsf{K} \Vert + \lambda_n \vert S^\star  \vert, \end{equation}
where we used both Eq.\! \eqref{cs} and the fact that $\W \mathsf{K} -  \mathsf{K} \W   =0$. Combining \eqref{mins1} and \eqref{min*}, we get
\[
 \mathbb P\{\widehat S \in \mathcal S_1 \} \leq \mathbb P \Big\{\min_{S \in \mathcal S_1} Q(S) \leq Q(S^\star ) \Big\} 
\leq \mathbb P \Big\{ \Vert \widehat{\mathsf{K}}-\mathsf{K} \Vert \geq \frac{c_0(S^\star) - \lambda_n \vert S^\star  \vert} 4 \Big\}\,.
\]
To control the term $\mathbb P(\widehat S \in \mathcal S_2 )$, we use that $\displaystyle \min_{S \in \mathcal S_2} Q(S) \geq \lambda_n \min_{S \in \mathcal S_2} \vert S \vert \geq \lambda_n (\vert S^\star  \vert +1)$. By Eq.\! \eqref{min*}, it follows
\begin{align*}
 \mathbb P\big\{\widehat S \in \mathcal S_2 \big\} &\leq \mathbb P \Big\{\min_{S \in \mathcal S_2} Q(S) \leq Q(S^\star ) \Big\}\\
 & \leq \mathbb P \Big\{ \lambda_n (\vert S^\star  \vert +1) \leq 2 \Vert \widehat{\mathsf{K}}- \mathsf{K} \Vert + \lambda_n \vert S^\star  \vert \Big\} \\
 & = \mathbb P \Big\{ \Vert \widehat{\mathsf{K}}- \mathsf{K} \Vert \geq \frac{\lambda_n}2 \Big\}\,.
\end{align*}
The proof of Theorem \ref{t:main1} follows directly by \eqref{a:conc}. The corollary is a direct consequence using Borel-Cantelli's Lemma.

\subsection{Proof of Theorem \ref{asymptbeta}}
Since $\Delta(\mathsf{K}) \Phi_S$ is of full rank, the value $\widehat \beta_S = \big( \Delta(\widehat{\mathsf{K}}) \Phi_S \big)^\dagger \Delta(\widehat{\mathsf{K}}) a_0$ is the unique solution to Eq.\! \eqref{hatbetaS} with probability tending to one asymptotically. Since the value of $\widehat \beta_S$ does not depend on $a_0 \in \mathcal A_S$, one can take $a_0 = w$ in view of $S^\star \subseteq S$. We obtain
$$ \widehat \beta_S = \big( \Delta(\widehat{\mathsf{K}}) \Phi_S \big)^\dagger \Delta(\widehat{\mathsf{K}}) w = - \big( \Delta(\widehat{\mathsf{K}}) \Phi_S \big)^\dagger \Delta(\mathsf{W}) \widehat k. $$
The result follows from Slutsky's lemma, using that $( \Delta(\widehat{\mathsf{K}}) \Phi_S)^\dagger$ converges in probability towards $( \Delta(K) \Phi_S)^\dagger$ and 
$$\sqrt n \, \big( \Delta(\W) \widehat k -  \Delta(\mathsf{W}) k \big) \xrightarrow[n \to \infty]{d} \mathcal N \big(0, \Delta(\mathsf{W}) \Sigma \Delta(\mathsf{W})^\top \big).$$

\bibliography{biblio}  
\end{document}